\documentclass[12pt]{article}
\usepackage{mycommands}
\usepackage{tikz}
\usetikzlibrary{decorations.pathreplacing}
\title{Constructing Reducible Brill--Noether Curves}

\begin{document}
\maketitle

\begin{abstract}
A fundamental problem in the theory of algebraic curves in projective space
is to understand which reducible curves arise as limits of smooth curves
of general moduli.
Special cases of this question and variants have been critical in the resolution
of many problems in the theory of algebraic curves
over the past half century; examples include 
Sernesi's proof of the existence of components of the Hilbert scheme with the expected number
of moduli when the Brill--Noether number is negative \cite{sernesi},
and Ballico's proof the Maximal Rank Conjecture for quadrics \cite{ball2}.

In this paper, we give close-to-optimal bounds on this problem
when the nodes are general points and the components are general in moduli.

The results given here significantly extend those cases
established by Sernesi \cite{sernesi}, Ballico \cite{ball2}, and others.
As explained in \cite{over}, they also play a key role in the author's proof
of the Maximal Rank Conjecture \cite{mrc}.
\end{abstract}

\section{Introduction}

A central problem in algebraic geometry is to understand aspects of the geometry of general curves
in projective space.
For properties preserved by deformation,
a powerful and flexible tool for this purpose is degeneration: We establish the desired property
at a reducible curve, and show this reducible curve may be deformed to a general smooth curve.
Degeneration techniques have enabled
the proof of many such results,
including Eisenbud and Harris's proof of the Brill--Noether theorem and related results 
via limit linear series \cite{eis-har},
Sernesi's proof of the existence of components of the Hilbert scheme with the expected number
of moduli when the Brill--Noether number is negative \cite{sernesi},
Gieseker's proof of the existence of space curves with specified degree and genus \cite{gie},
and (using the present work) the Maximal Rank Conjecture \cite{mrc}.

The goal of the present paper is to make such arguments
easier by systematically studying which reducible curves
in projective space can be deformed to general smooth curves:

\begin{quest*} If $f \colon C_1 \cup_\Gamma C_2 \to \pp^r$
is a map from a reducible curve, under what conditions can $f$ be deformed to an
immersion of a general smooth curve?
\end{quest*}

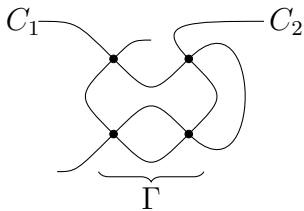
\begin{center}
\begin{tikzpicture}
\filldraw (1, 1) circle [radius=0.05];
\filldraw (2, 1) circle [radius=0.05];
\filldraw (1, 2) circle [radius=0.05];
\filldraw (2, 2) circle [radius=0.05];
\draw (0, 2.5) .. controls (0.5, 2.5) .. (1, 2);
\draw (0.25, 0.5) .. controls (0.5, 0.5) .. (1, 1);
\draw (1, 1) .. controls (1.5, 1.5) .. (2, 1);
\draw (1, 2) .. controls (1.5, 1.5) .. (2, 2);
\draw (2, 1) .. controls (3, 0) and (3, 3) .. (2, 2);
\draw (3, 2.5) .. controls (2, 2.5) and (1.5, 2.5) .. (2, 2);
\draw (2, 2) .. controls (2.5, 1.5) .. (2, 1);
\draw (2, 1) .. controls (1.5, 0.5) .. (1, 1);
\draw (1, 1) .. controls (0.5, 1.5) .. (1, 2);
\draw (1, 2) .. controls (1.25, 2.25) .. (1.5, 2.25);
\draw (-0.2, 2.5) node{$C_1$};
\draw (3.3, 2.5) node{$C_2$};
\draw [decorate, decoration={brace, mirror, amplitude=0.75ex}] (0.8, 0.5) -- (2.2, 0.5);
\draw (1.5, 0.15) node{$\Gamma$};
\end{tikzpicture}
\end{center}

In the present paper, we focus on the case where the ``pieces'' of $f$ are general, namely
when $C_1$ and $C_2$ are of general moduli and $f(\Gamma)$ is a general set of points in $\pp^r$.
We also consider variants where $f(\Gamma)$ is a general set of points in
some linear space $\Lambda \subset \pp^r$.

To fix notation,
write $\bar{M}_g(\pp^r, d)$ for Kontsevich's space of stable maps 
$C \to \pp^r$ of degree~$d$, from a nodal curve $C$
of genus $g$.
There is a natural map
$\bar{M}_g(\pp^r, d) \to \bar{M}_g$.

\begin{defi}
We refer to a stable map $C \to \pp^r$
as a
\emph{Weak Brill--Noether curve} (\emph{WBN-curve}) if it
corresponds to a point in a component of
$\bar{M}_g(\pp^r, d)$ which both
dominates $\bar{M}_g$,
and whose generic member is a map
from a smooth curve which is either nondegenerate or nonspecial, and which is an immersion if $r \geq 3$,
birational onto its image if $r = 2$, and finite if $r = 1$.

In the former case, we refer to it as a \emph{Brill--Noether curve} (\emph{BN-curve});
in the later case, we refer to it as a \emph{limit NonSpecial curve} (\emph{NS-curve});
and if is both nonspecial and nondegenerate,
we refer to it as a
\emph{Nondegenerate NonSpecial curve} (\emph{NNS-curve}).
Note that a general BN-curve is an NNS-curve if and only if $d \geq g + r$.

Additionally, we say a stable map $f \colon C \to \pp^r$
is \emph{limit linearly normal} if it is a limit of linearly normal stable maps.
Note that a general BN-curve is limit linearly normal if and only if $d \leq g + r$.

Finally, we say a stable map $f \colon C \to \pp^r$
is an \emph{interior curve} if it lies in a unique component of the
corresponding space of stable maps.
\end{defi}

The Brill--Noether theorem
asserts that BN-curves of degree $d$ and genus $g$ in $\pp^r$ exist
if and only if the \emph{Brill--Noether number}
\[\rho(d, g, r) := (r + 1)d - rg - r(r + 1) \geq 0;\]
and that in this case, the locus of BN-curves forms an irreducible
component of $\bar{M}_g(\pp^r, d)$.

\medskip

Returning to our main question: In order for $f$ to be deformable to an immersion
of a general smooth curve, it is obviously necessary for an immersion of that degree
of a general smooth curve of that genus to exist.
The natural conjecture would be that this is sufficient:

\begin{conj} \label{conjrbn} Let $f \colon C_1 \cup_\Gamma C_2 \to \pp^r$ be a stable map from a reducible curve,
such that $f|_{C_1}$ and $f|_{C_2}$ are BN-curves, and $f(\Gamma)$ is a general set of $n = \# \Gamma$ points in $\pp^r$.
Then $f$ is a BN-curve if and only if it has nonnnegative Brill--Noether number.
\end{conj}

Implicit in the formulation of this conjecture are a couple of inequalities
in the degree $d_i$ and genus $g_i$ of $f|_{C_i}$.
Indeed, since
the curve $f$ is of degree $d_1 + d_2$ and genus $g_1 + g_2 + n - 1$,
the condition of having
nonnegative Brill--Noether number is just
\[(r + 1)(d_1 + d_2) - r(g_1 + g_2 + n - 1) - r(r + 1) \geq 0.\]

Moreover, since $f(\Gamma)$ is a general set of $n$ points in $\pp^r$,
the natural maps
$\bar{M}_{g_i,n}^\circ(\pp^r, d_i) \to (\pp^r)^n$ must be
dominant. In particular we must have
\[(r + 1)d_i - (r - 3)(g_i - 1) + n = \dim \bar{M}_{g_i,n}^\circ(\pp^r, d_i) \geq \dim (\pp^r)^n = rn,\]
or upon rearrangement,
\begin{equation} \label{nmax}
(r + 1)d_i - (r - 3)(g_i - 1) - (r - 1)n \geq 0.
\end{equation}

%
%

\medskip

\noindent
The difficulty of Conjecture~\ref{conjrbn} is dictated by the parameter $n$:
\begin{itemize}
\item When $n$ is small relative to $r$, the automorphism group $\aut \pp^r$
acts transitively (or close to transitively) on the possible sets $f(\Gamma)$.
It is thus trivial (or relatively easy) to see that $f$ is a BN-curve.

Already, instances of this conjecture in or close to this easier range --- or variants
when $f(\Gamma)$ is general in a linear space $\Lambda \subset \pp^r$ --- have
had many applications; examples include:

\begin{itemize}
\item In \cite{severi}, Sernesi proves a variant of Conjecture~\ref{conjrbn}
when $f(\Gamma)$ is general in a hyperplane, in the special case where
$f|_{C_2}$ is a rational normal curve in a hyperplane and $n \leq r + 2$.

Using this, he deduces the existence of components of the Hilbert scheme with the expected
number of moduli when the Brill--Noether number is negative.

\item In \cite{ball2}, Ballico proves a variant of Conjecture~\ref{conjrbn}
when $f(\Gamma)$ is general in a hyperplane, in the special case where
$f|_{C_1}$ is an elliptic normal curve and $n = r + 1$.

Using this, he deduces the Maximal Rank Conjecture for quadrics.
\end{itemize}

\item When $n$ is large relative to $r$, but still small relative to the maximum
possible given the bounds \eqref{nmax}, Conjecture~\ref{conjrbn} is of intermediate difficulty.

\item As $n$ approaches the maximum possible value given the bounds \eqref{nmax},
Conjecture~\ref{conjrbn} becomes more difficult.

One reason for this difficulty is the existence of counterexamples
when the bounds \eqref{nmax} are achieved.
Namely, if $C_1 = C_2$, and $f|_{C_1} = f|_{C_2}$, and \eqref{nmax} is an equality,
it is an easy exercise that Conjecture~\ref{conjrbn} is \emph{false}.

Developing techniques that work in this regime is important for applications.
Indeed, the cases of Conjecture~\ref{conjrbn} --- or more precisely its analog
when $f(\Gamma)$ is general in a hyperplane --- with $n$ close to the maximum
possible given the bounds \eqref{nmax} are critical in the author's proof of the
Maximal Rank Conjecture \cite{mrc}. 
\end{itemize}

The central goal of the present paper is to develop a flexible technique
to study cases of Conjecture~\ref{conjrbn} that works
when the bounds \eqref{nmax} are close to an equality.
In light of the counterexamples mentioned above when \eqref{nmax} is achieved, the best
one might hope for is:

\begin{conj} \label{conjrbn-refined}
Conjecture~\ref{conjrbn} holds as long as \eqref{nmax} is not an equality
for some component, i.e.\ so long as for at least one $i \in \{1, 2\}$ we have
\[(r + 1) d_i - (r - 3)(g_i - 1) - (r - 1)n \geq 1.\]
\end{conj}

Our main theorem establishes Conjecture~\ref{conjrbn} for NNS-curves, so long as
\eqref{nmax} is at least $4$ away from an equality for one component, i.e.\ so long as
\[(r + 1) d_i - (r - 3)(g_i - 1) - (r - 1)n \geq 4.\]
This is close to optimal, in the sense that it would be false if $4$ were replaced by $0$.

We actually give a slightly stronger statement below since this slightly stronger
version is more useful as an inductive hypothesis:

\begin{thm} \label{main}
Let $C_i \to \pp^r$ (for $i \in \{1, 2\}$) be NNS-curves
of degree $d_i$ and genus $g_i$, which pass through
a set $\Gamma \subset \pp^r$ of $n \geq 1$ general points.
Suppose either that both curves are limit linearly normal
(equivalently $d_i = g_i + r$ for both $i \in \{1, 2\}$),
or alternatively that for at least one $i \in \{1, 2\}$
we have
\begin{equation} \label{addcond}
(r + 1) d_i - (r - 3)(g_i - 1) - (r - 1)n \geq \begin{cases} 
2 & \text{if $d_i > g_i + r$;} \\
4 & \text{if $d_i = g_i + r$.}
\end{cases}
\end{equation}
Then
$C_1 \cup_\Gamma C_2 \to \pp^r$ is a BN-curve, provided it has nonnegative Brill--Noether number.

Furthermore, if both $C_i \to \pp^r$ are general in some component of the space of NNS-curves
passing through $\Gamma$, then $C_1 \cup_\Gamma C_2 \to \pp^r$ is an interior BN-curve.
\end{thm}

(Note that if $d_i = g_i + r$ for both $i \in \{1, 2\}$,
the condition that $C_1 \cup_\Gamma C_2 \to \pp^r$ has nonnegative Brill--Noether number prevents \eqref{nmax} from being an equality.)

\medskip

The restriction to NNS-curves (as opposed to BN-curves)
arises only due to the dependence of Theorem~\ref{main} on results
of \cite{aly} on the interpolation problem:

\begin{thm}[Corollary~1.4 of \cite{aly}] \label{cor14}
There exists an NNS-curve $C \to \pp^r$ of degree $d$ and genus $g$
to $\pp^r$ (with $d \geq g + r$),
passing through $n$ general points,
if and only if
\[\begin{cases}
(r - 1) n \leq (r + 1) d - (r - 3)(g - 1) & \text{if $(d, g, r) \notin \{(5, 2, 3), (7, 2, 5)\}$;} \\
\phantom{(r - 1)} n \leq 9 & \text{if $(d, g, r) \in \{(5, 2, 3), (7, 2, 5)\}$.} \\
\end{cases}\]
\end{thm}

If an analog of this result were known for all BN-curves,
the method developed in the present paper would apply
in that situation as well to prove an analog of Theorem~\ref{main}
for all BN-curves.

\medskip

The most basic approach --- used, for example, by Sernesi in \cite{sernesi} --- to proving cases of Conjecture~\ref{conjrbn}
is to calculate the fiber dimension of the map from the space of
stable maps to the moduli space of curves
at the given reducible curve, thereby showing it lies in a component dominating
the moduli space of curves and hence is a BN-curve.
Unfortunately, such an approach is extremely difficult for $n$ large --- and impossible to use
in variants where $f(\Gamma)$ is general in a linear space for $n$ large
(in which case the fiber dimension is usually provably incorrect).

Instead, we prove Theorem~\ref{main} by an inductive argument showing
that such curves lie in the same component as
\emph{another} curve which we know is a BN-curve by calculation of the fiber dimension.
Rather than finding an irreducible curve in the space of maps,
the key insight here is to draw
a ``broken arc'' (iteratively specialize and then deform)
in the space of stable maps, connecting these two points of the moduli space:

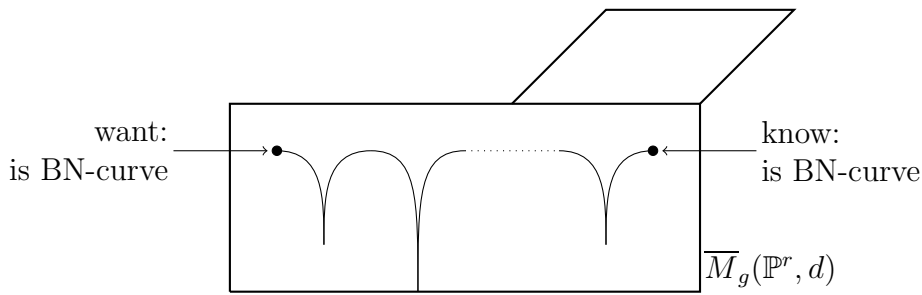
\begin{center}
\begin{tikzpicture}[scale=1.25]
\draw[thick] (0, 0) -- (0, 2) -- (5, 2) -- (5, 0) -- (0, 0);
\draw[thick] (3, 2) -- (4, 3) -- (6, 3) -- (5, 2);
\filldraw (0.5, 1.5) circle [radius=0.05];
\draw (0.5, 1.5) .. controls (1, 1.5) and (1, 1) .. (1, 0.5);
\draw (1, 0.5) .. controls (1, 1) and (1, 1.5) .. (1.5, 1.5);
\draw (1.5, 1.5) .. controls (2, 1.5) and (2, 1) .. (2, 0);
\draw (2, 0) .. controls (2, 1) and (2, 1.5) .. (2.5, 1.5);
\draw[dotted] (2.5, 1.5) -- (3.5, 1.5);
\draw (3.5, 1.5) .. controls (4, 1.5) and (4, 1) .. (4, 0.5);
\draw (4, 0.5) .. controls (4, 1) and (4, 1.5) .. (4.5, 1.5);
\filldraw (4.5, 1.5) circle [radius=0.05];
\draw[->] (-0.6, 1.5) -- (0.4, 1.5);
\draw (-1.5, 1.5) node[align=right]{want: \\ is BN-curve};
\draw[<-] (4.6, 1.5) -- (5.6, 1.5);
\draw (6.5, 1.5) node[align=left]{know: \\ is BN-curve};
\draw (5.75, 0.25) node{$\bar{M}_{g}(\pp^r, d)$};
\end{tikzpicture}
\end{center}

Provided we check the specializations are to smooth points of the space of stable maps,
this shows our given such reducible curve is in the same component as the other curve, and is thus
a BN-curve as desired.

To carry out this approach, we must first establish the base cases
for our induction --- i.e.\ we must show
that certain other curves where we can calculate
the fiber dimension (``know: is BN-curve'' in the above diagram) are BN-curves.
For this, we use results on interpolation for restricted tangent bundles
of general curves \cite{tan}, which we show in Section~\ref{sec:tangent} gives a tool
to study Conjecture~\ref{conjrbn} ``one component at a time''.
While the assumptions that both components have nonnegative Brill--Noether numbers
reference only one component at a time\ldots
\begin{align*}
(r + 1) d_1 - r g_1 - r(r + 1) &\geq 0 \\
(r + 1) d_2 - r g_2 - r(r + 1) &\geq 0 \\
\intertext{\ldots the condition that the union has nonnegative Brill--Noether number
references both components:}
(r + 1) (d_1 + d_2) - r(g_1 + g_2 + n - 1) - r(r + 1) &\geq 0.
\end{align*}
However, subtracting the final two conditions, we see
that the final condition follows from one involving only the first component:
\[(r + 1) d_1 - r g_1 + r \geq rn.\]
It is exactly in this regime that we can verify the base case of our inductive
argument using these techniques.
Since this step does not depend on results of \cite{aly}, it can be done even when the
components are special; for this reason we state it here as a separate theorem:

\begin{thm} \label{main-sp}
Let $C_i \to \pp^r$ (for $i \in \{1, 2\}$) be WBN-curves
of degree $d_i$ and genus $g_i$, which pass through
a set $\Gamma \subset \pp^r$ of $n \geq 1$ general points.
Suppose that, for at least one $i \in \{1, 2\}$, we have
\[(r + 1) d_i - r g_i + r \geq rn.\]
Then
$C_1 \cup_\Gamma C_2 \to \pp^r$ is a WBN-curve.

Furthermore, if both $C_i \to \pp^r$ are general in some component of the space of WBN-curves
passing through $\Gamma$, then $C_1 \cup_\Gamma C_2 \to \pp^r$ is an interior WBN-curve.
\end{thm}

These arcs are constructed by further specializing one of the components, say $C'$,
to a reducible curve $C'_1 \cup D_1'$; this results in a specialization of $C' \cup C''$ given by
\[(C'_1 \cup D_1') \cup C'' = C_1' \cup (D_1' \cup C'').\]
We then deform $D_1' \cup C''$ to a smooth curve $C_1''$.
Finally, we iterate this procedure, alternating between components
(next we would specialize $C_1''$ --- to a different reducible curve, not back to $D'_1 \cup C''$):

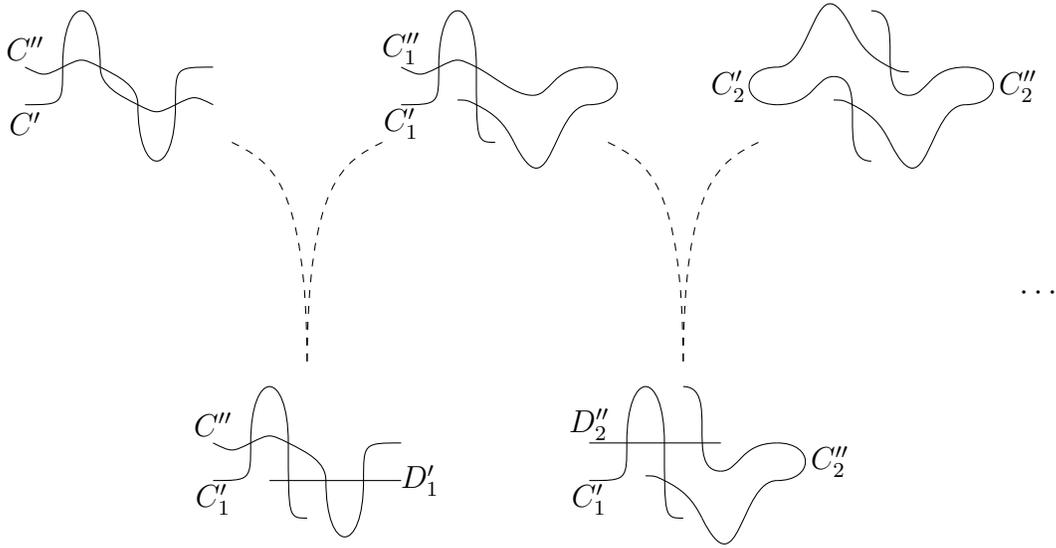
\begin{center}
\begin{tikzpicture}[scale=0.5]
\draw (0, -0.5) node{$C'$};
\draw (0, 1.5) node{$C''$};
\draw (1, 1) .. controls (1, 3) and (2, 3) .. (2, 1);
\draw (1, 1) .. controls (1, 0) .. (0, 0);
\draw (2, 1) .. controls (2, 0.5) and (2.5, 0.25) .. (3, 0);
\draw (3, 0) .. controls (3.5, -0.25) .. (4, 0);
\draw (4, 0) .. controls (4.5, 0.25) .. (5, 0);
\draw (4, 0) .. controls (4, -2) and (3, -2) .. (3, 0);
\draw (4, 0) .. controls (4, 1) .. (5, 1);
\draw (3, 0) .. controls (3, 0.5) and (2.5, 0.75) .. (2, 1);
\draw (2, 1) .. controls (1.5, 1.25) .. (1, 1);
\draw (1, 1) .. controls (0.5, 0.75) .. (0, 1);
\begin{scope}[shift={(5, -10)}]
\draw (0, -0.5) node{$C_1'$};
\draw (0, 1.5) node{$C''$};
\draw (5.5, 0) node{$D_1'$};
\draw (1, 1) .. controls (1, 3) and (2, 3) .. (2, 1);
\draw (1, 1) .. controls (1, 0) .. (0, 0);
\draw (2, 1) .. controls (2, -1) .. (2.5, -1);
\draw (1.5, 0) -- (5, 0);
\draw (4, 0) .. controls (4, -2) and (3, -2) .. (3, 0);
\draw (4, 0) .. controls (4, 1) .. (5, 1);
\draw (3, 0) .. controls (3, 0.5) and (2.5, 0.75) .. (2, 1);
\draw (2, 1) .. controls (1.5, 1.25) .. (1, 1);
\draw (1, 1) .. controls (0.5, 0.75) .. (0, 1);
\end{scope}
\begin{scope}[shift={(10, 0)}]
\draw (0, -0.5) node{$C_1'$};
\draw (0, 1.5) node{$C_1''$};
\draw (1, 1) .. controls (1, 3) and (2, 3) .. (2, 1);
\draw (1, 1) .. controls (1, 0) .. (0, 0);
\draw (2, 1) .. controls (2, -1) .. (2.5, -1);
\draw (3.5, 0.25) .. controls (3, 0.25) and (2.5, 0.75) .. (2, 1);
\draw (3.5, 0.25) .. controls (4, 0.25) and (4, 1) .. (5, 1);
\draw (5, 1) .. controls (6, 1) and (6, 0) .. (5, 0);
\draw (5, 0) .. controls (4, 0) and (4, -3) .. (3, -1);
\draw (3, -1) .. controls (2.75, -0.5) and (2.5, -0.25) .. (2, 0);
\draw (2, 0) .. controls (1.75, 0.125) .. (1.5, 0.125);
\draw (2, 1) .. controls (1.5, 1.25) .. (1, 1);
\draw (1, 1) .. controls (0.5, 0.75) .. (0, 1);
\end{scope}
\begin{scope}[shift={(15, -10)}]
\draw (0, -0.5) node{$C_1'$};
\draw (0, 1.5) node{$D_2''$};
\draw (6.4, 0.5) node{$C_2''$};
\draw (1, 1) .. controls (1, 3) and (2, 3) .. (2, 1);
\draw (1, 1) .. controls (1, 0) .. (0, 0);
\draw (2, 1) .. controls (2, -1) .. (2.5, -1);
\draw (3.5, 0.25) .. controls (2.5, 0.25) and (3.5, 2.5) .. (2.5, 2.5);
\draw (3.5, 1) -- (0, 1);
\draw (3.5, 0.25) .. controls (4, 0.25) and (4, 1) .. (5, 1);
\draw (5, 1) .. controls (6, 1) and (6, 0) .. (5, 0);
\draw (5, 0) .. controls (4, 0) and (4, -3) .. (3, -1);
\draw (3, -1) .. controls (2.75, -0.5) and (2.5, -0.25) .. (2, 0);
\draw (2, 0) .. controls (1.75, 0.125) .. (1.5, 0.125);
\end{scope}
\begin{scope}[shift={(20, 0)}]
\draw (6.4, 0.5) node{$C_2''$};
\draw (-1.3, 0.5) node{$C_2'$};
\draw (1.5, 0.75) .. controls (2.5, 0.75) and (1.5, -1.5) .. (2.5, -1.5);
\draw (1.5, 0.75) .. controls (1, 0.75) and (1, 0) .. (0, 0);
\draw (0, 0) .. controls (-1, 0) and (-1, 1) .. (0, 1);
\draw (0, 1) .. controls (1, 1) and (1, 4) .. (2, 2);
\draw (2, 2) .. controls (2.25, 1.5) and (2.5, 1.25) .. (3, 1);
\draw (3, 1) .. controls (3.25, 0.875) .. (3.5, 0.875);
\draw (3.5, 0.25) .. controls (2.5, 0.25) and (3.5, 2.5) .. (2.5, 2.5);
\draw (3.5, 0.25) .. controls (4, 0.25) and (4, 1) .. (5, 1);
\draw (5, 1) .. controls (6, 1) and (6, 0) .. (5, 0);
\draw (5, 0) .. controls (4, 0) and (4, -3) .. (3, -1);
\draw (3, -1) .. controls (2.75, -0.5) and (2.5, -0.25) .. (2, 0);
\draw (2, 0) .. controls (1.75, 0.125) .. (1.5, 0.125);
\end{scope}
\draw[dashed] (5.5, -1) .. controls (6.5, -1.5) and (7.5, -2) .. (7.5, -7);
\draw[dashed] (9.5, -1) .. controls (8.5, -1.5) and (7.5, -2) .. (7.5, -7);
\begin{scope}[shift={(10, 0)}]
\draw[dashed] (5.5, -1) .. controls (6.5, -1.5) and (7.5, -2) .. (7.5, -7);
\draw[dashed] (9.5, -1) .. controls (8.5, -1.5) and (7.5, -2) .. (7.5, -7);
\end{scope}
\draw (27, -5) node{$\cdots$};
\end{tikzpicture}
\end{center}

Note that even if 
$C'$ and $C''$ do not meet at any additional
point not in $\Gamma$, and have distinct tangent directions at the points of $\Gamma$ --- so that $f$ is the natural immersion of the scheme-theoretic union --- this broken arc may still not make sense in the Hilbert scheme
compactification, so it is important to
work in the space of stable maps even in this case.

\medskip

While almost all of the remainder of the paper is consumed by the proof of
our main Theorem~\ref{main},
we also study several special cases of a variant of Conjecture~\ref{conjrbn} where
$f(\Gamma)$ is general in a hyperplane or other linear subspace.
A more systematic study of this variant --- using techniques developed in the present paper --- is
deferred to another paper \cite{rbn2},
since additional results on the interpolation problem obtained in a sequence of papers
by the author and others \cite{quadrics, p4, ibe, vogt} are required,
and this sequence of papers uses results of the present paper.
The special cases examined in the present paper
are chosen for a combination of the following reasons:
\begin{enumerate}
\item They can easily be obtained with the methods developed here to prove Theorem~\ref{main}.
\item They have applications to several other geometric problems, including
to the above-mentioned sequence of papers on the interpolation problem,
and to the Maximal Rank Conjecture.
\end{enumerate}

In this direction, we first note that the argument used to establish
Theorem~\ref{main-sp} also establishes, with no additional work, the following
special case of Conjecture~\ref{conjrbn} where $f(\Gamma)$ is general in a hyperplane:

\begin{thm} \label{main-hyp}
Let $H \subset \pp^r$ be a hyperplane,
and let $C \to \pp^r$ be a WBN-curve and $D \to H \hookrightarrow \pp^r$ be an NS-curve,
with $C$ transverse to $H$, 
which pass through
a set $\Gamma \subset H$ of $n \geq 1$ general points in $H$.
Suppose that, writing $d'$ and $g'$
for the degree and genus of $C$,
we have
\[d' - rg' - 1 \geq 0 \tand (r + 1) d' - r g' + r \geq rn.\]
Then
$C \cup_\Gamma D \to \pp^r$ is a WBN-curve.

Furthermore, if both $C \to \pp^r$ and $D \to H$
are general in some component of the space of WBN-curves (respectively NS-curves)
passing through $\Gamma$, then $C \cup_\Gamma D \to \pp^r$ is an interior WBN-curve.
\end{thm}

We next consider the variant of Conjecture~\ref{conjrbn} where $f(\Gamma)$ is a set of
general points in a linear space $\Lambda \subset \pp^r$ of smaller dimension $a \leq r$.
As discussed earlier, the difficulty of Conjecture~\ref{conjrbn} rises with $n$,
so the easiest nontrivial case here is when $n = a + 2$ (if $n = a + 1$
the points are general in $\pp^r$).
In this setting, we consider the case when one component is a nondegenerate
curve of degree $d$ and genus $g$,
and the other is a rational normal curve contained in $\Lambda$.
If the first component is general, then we must be in the range
where general curves admit $n$-secant $a$-planes; if $n = a + 2$,
this condition is exactly $a \geq \frac{r - 2}{2}$.
Moreover, we must be in the range where the reducible curve
has nonnegative Brill--Noether number;
if $n = a + 2$,
this condition is exactly $a \geq r - \rho(d, g, r)$.
In short, a necessary condition is
\[\max\left(\frac{r - 2}{2}, r - \rho(d, g, r)\right) \leq a \leq r.\]
The following result essentially shows this condition is sufficient:

\begin{thm} \label{small-mid}
Let $f \colon C \to \pp^r$ be a general BN-curve of degree $d$ and genus $g$,
and $a$ be an integer with
\[\max\left(\frac{r - 2}{2}, r - \rho(d, g, r)\right) \leq a \leq r.\]
Then there exists an interior BN-curve $\hat{f} \colon C \cup_\Gamma \pp^1 \to \pp^r$
with $\# \Gamma = a + 2$ and $\hat{f}|_{\pp^1}$ of degree $a$, such that
$\hat{f} |_C = f$, and such that the image of $2\Gamma \subset C$ under $f$ spans $\pp^r$.
\end{thm}

Such degenerations are used in \cite{ibe} to give
bounds on the number of general points contained in a hyperplane section of a general BN-curve:

\begin{thm*}[Theorem~1.5 of \cite{ibe}]
The hyperplane section of a general BN-curve of degree $d$ and genus $g$ in $\pp^r$
contains $d - n$ general points (with $0 \leq n \leq d$) if
\[(2r - 3)(d + 1) - (r - 2)^2 (g - n) - 2r^2 + 3r - 9 \geq 0.\]
\end{thm*}

Finally, we consider the case where
$f(\Gamma)$ is general in a hyperplane, with one component $f|_C$ nondegenerate, and the other component $f|_D$
contained in the hyperplane, when $n$ is small.
This will form the base case for our later systematic study of this variant of Conjecture~\ref{conjrbn} in \cite{rbn2}
mentioned above, as well as having direct application (via Corollary~\ref{cor:small-hyp}) to cases of the Maximal
Rank Conjecture with small genus in \cite{mrc}.

Since $f|_D$ may be both special and degenerate,
it may not be possible to deform $f$ to a
map from a smooth curve (let alone one from a smooth curve of general moduli!).
As in the proof of Corollary~4.3
of \cite{hh} mutatis mutandis,
$f$ admits a first-order deformation away from the locus of curves with reducible
source if and only if:
\begin{equation} \label{nC}
n + d'' - g'' - r = n - (\dim H^1 (N_{f|_D}) + 1) \geq 0;
\end{equation}
we therefore focus on cases where this inequality holds.
The following result establishes this variant of Conjecture~\ref{conjrbn} when $n$ is small
($n \leq r + 2$):

\begin{thm} \label{small-hyp}
Let $H \subset \pp^r$ be a hyperplane, $\Gamma \subset H$
be a set of $n \geq 1$ general points,
$f \colon C \to \pp^r$ be a WBN-curve passing through $\Gamma$, and $g \colon D \to H$
be a BN-curve passing through $\Gamma$, with 
$f$ transverse to $H$ along $\Gamma$.
Write $g''$ for the genus of $D$ and $d''$ for the degree of $g$.
If
\[n \leq r + 2 \tand d'' + n \geq g'' + r,\]
then
$C \cup_\Gamma D \to \pp^r$ is a BN-curve provided that it
has nonnegative Brill--Noether number.

Furthermore, if both $C \to \pp^r$ and $D \to H$
are general in some component of the space of WBN-curves (respectively BN-curves)
passing through $\Gamma$, then $C \cup_\Gamma D \to \pp^r$ is an interior BN-curve.
\end{thm}

\begin{rem} If $g$ is instead a WBN-curve but not a BN-curve,
then $C \cup_\Gamma D \to \pp^r$ is a WBN-curve
by Theorem~\ref{main-sp}.
\end{rem}

Finally, we state an (easy) corollary of the above theorem,
where $C$ is replaced with the union of a WBN-curve and disjoint lines.
Such degenerations are useful in the proof of the Maximal Rank Conjecture
when $g$ is small (a regime in which degeneration to any reducible
curve with only a bounded number of components seems futile for numerical reasons).

\begin{cor} \label{cor:small-hyp}
Let $H \subset \pp^r$ be a hyperplane, $\Gamma \subset H$
be a set of $n \geq 1$ general points, $\{p_1, p_2, \ldots, p_m\} \subset H$
be an independently general set of $m \geq 0$ points,
$f \colon C \to \pp^r$ be a WBN-curve passing through $\Gamma$, $g \colon D \to H$
be a BN-curve passing through $\Gamma \cup \{p_1, p_2, \ldots, p_m\}$, and $h_i \colon \pp^1 \to \pp^r$
be lines passing through $p_i$ (for $1 \leq i \leq m$), with $f$ transverse to $H$ along $\Gamma$
and $h_i$ transverse to $H$ along $p_i$.
Write $g''$ for the genus of $D$ and $d''$ for the degree of $g$.
If
\[n \leq r + 2 \tand d'' + n + m \geq g'' + r,\]
then
$C \cup_\Gamma D \cup_{p_1} \pp^1 \cup_{p_2} \pp^1 \cdots \cup_{p_m} \pp^1 \to \pp^r$ is a BN-curve provided that it
has nonnegative Brill--Noether number.

Furthermore, if $C \to \pp^r$ and $D \to H$ and all $h_i \colon \pp^1 \to \pp^r$
are general in some component of the space of WBN-curves (respectively BN-curves, respectively lines)
passing through $\Gamma$ (respectively $\Gamma \cup \{p_1, p_2, \ldots, p_m\}$, respectively $p_i$),
then $C \cup_\Gamma D \cup_{p_1} \pp^1 \cup_{p_2} \pp^1 \cdots \cup_{p_m} \pp^1 \to \pp^r$ is an interior BN-curve.
\end{cor}

As mentioned earlier,
the results and techniques
of this paper have already found application to several geometric
problems via degeneration arguments. These
include various generalizations of Theorem~\ref{cor14} discussed above
in \cite{quadrics, p4, ibe, vogt},
as well as the proof of Severi's 1915 Maximal Rank Conjecture in \cite{mrc}:

\begin{thm*}[Maximal Rank Conjecture \cite{mrc}]
If $C \subset \pp^r$ is a general BN-curve ($r \geq 3$), the restriction maps
\[H^0(\oo_{\pp^r}(k)) \to H^0(\oo_C(k))\]
are of maximal rank (i.e.\ either injective or surjective).

Or equivalently, the dimension of the space of polynomials
of degree $k$ which vanish on $C$ is given by
\[\begin{cases}
\binom{r + k}{k} - (kd + 1 - g) & \text{if $kd + 1 - g \leq \binom{r + k}{k}$ and $k \geq 2$;} \\
0 & \text{otherwise.}
\end{cases}\]
\end{thm*}

\paragraph{Note 1:} Throughout this paper, we work over an algebraically
closed field of characteristic zero.

\paragraph{Note 2:} Since any specialization of a BN-curve is a BN-curve,
we may suppose all curves $C_i \to \pp^r$, $f$, $g$, etc.\ appearing in our theorem statements above
are general in some component of the space of such curves
(e.g.\ in Theorem~\ref{small-hyp}, we may suppose $g$ is general in some component
of the space of curves to $H$ passing through $\Gamma$).
We will show in this case that the resulting stable maps from reducible
curves are interior BN-curves or WBN-curves as indicated.

\subsection*{Acknowledgements}

The author would like to thank Joe Harris for
his guidance throughout this research.
The author would also like to thank
Izzet Coskun, Gavril Farkas, Davesh Maulik,
Dmitry Vaintrob, Isabel Vogt,
and members of the Harvard and MIT mathematics departments,
for helpful conversations;
and to acknowledge the generous
support both of the Fannie and John Hertz Foundation,
and of the Department of Defense
(NDSEG fellowship).

\section{WBN-curves and the Restricted Tangent Bundle \label{sec:tangent}}

In this section we give a criterion for $f \colon  C \to \pp^r$
to be a WBN-curve, in terms of the \emph{restricted tangent bundle}
$f^* T_{\pp^r}$ (also sometimes called the ``Lazarsfeld bundle'' due its usage
by Lazarsfeld in the study of syzygies of curves \cite{laz}).
We then use this condition to prove Theorems~\ref{main-sp} and~\ref{main-hyp}.

\begin{lm} \label{exp-dim}
Any component of $\bar{M}_{g, n}(\pp^r, d)$ whose general member is a
WBN-curve is generically reduced of the expected dimension $(r + 1)d - (r - 3)(g - 1) + n$.
\end{lm}
\begin{proof}
Write $r'$
for the dimension of the linear span of the general member $f \colon C \to \pp^r$.
Our component is then birational to a bundle
over the Grassmanian $\operatorname{Gr}(r', r)$,
whose generic fiber is the component $N \subseteq \bar{M}_{g, n}(\pp^{r'}, d)$
corresponding to BN-curves.

By results of Griffiths, Harris, and Gieseker \cite{gp, bn},
the component $N$ is generically reduced of the expected dimension
$(r' + 1)d - (r' - 3)(g - 1) + n$.

If $r' = r$ this completes the proof. Otherwise, since $f$ is general,
we claim it must be linearly normal; indeed, it suffices to show that if $f$
was not linearly normal, it admits a deformation whose linear span is $(r' + 1)$-dimensional.
For this, choose coordinates $[x_0 : x_1 : \cdots : x_r]$ on $\pp^r$ so the
linear span of $f$ is defined by $x_{r' + 1} = x_{r' + 2} = \cdots = x_r = 0$, and
write $f$ as $[f_0 : f_1 : \cdots : f_k : 0 : \cdots : 0]$
for $f_i \in H^0(\mathcal{L})$ for some line bundle $\mathcal{L}$.
If $f$ were not linearly normal, we could pick some $f_{k + 1} \in H^0(\mathcal{L}) \smallsetminus \langle f_1, f_2, \ldots, f_k \rangle$;
then $[f_0 : f_1 : \cdots : f_k : \lambda f_{k + 1} : 0 \cdots 0]$ for $\lambda$ generic provides
the required deformation of $f$.
We conclude $f$ is linearly normal as claimed.

The definition of WBN-curves then implies $d = g + r'$,
so our component is generically reduced of dimension
\begin{align*}
\dim \operatorname{Gr}(r', r) + \dim N &= (r' + 1)(r - r') + [(r' + 1)d - (r' - 3)(g - 1) + n] \\
&= (d - g + 1)(r - d + g) + (d - g + 1)d - (d - g - 3)(g - 1) + n\\
&= (r + 1) d - (r - 3)(g - 1) + n,
\end{align*}
which is the expected dimension.
\end{proof}

\begin{lm} \label{wbn}
Let $f \colon C \to \pp^r$ be a curve, and $\Gamma \subset C$ a (possibly empty)
set of points whose images under $f$ are linearly independent.
If $f$ is a general WBN-curve, $H^1(f^* T_{\pp^r}(-\Gamma)) = 0$.
Conversely, if $H^1(f^* T_{\pp^r}(-\Gamma)) = 0$, then $f$ is an interior WBN-curve.
\end{lm}
\begin{proof}
First we consider the case when $\Gamma = \emptyset$.
Assume first that f is a general WBN curve.
By Lemma~\ref{exp-dim}, the component of $\bar{M}_g(\pp^r, d)$
containing $[f]$ is smooth at $[f]$ of the expected dimension.
Since this component by definition dominates $\bar{M}_g$,
the vertical tangent space $H^0(f^* T_{\pp^r})$ of
$\bar{M}_g(\pp^r, d) \to \bar{M}_g$ at $[f]$ is of the expected dimension
$\chi(f^* T_{\pp^r})$; thus $H^1(f^* T_{\pp^r}) = 0$.

For the converse, $H^1(f^* T_{\pp^r}) = 0$ implies the map
$\bar{M}_g(\pp^r, d) \to \bar{M}_g$ is smooth at $[f]$, so $[f]$ lies in a unique component
of $\bar{M}_g(\pp^r, d)$, which dominates $\bar{M}_g$.
If $f$ fails to be nondegenerate, then writing $\Lambda$ for its linear span, the exact sequence
\[0 \to f^* T_\Lambda \to f^* T_{\pp^r} \to f^* N_\Lambda \simeq \oo_C(1)^{\operatorname{codim} \Lambda} \to 0\]
implies $H^1(\oo_C(1)) = 0$.

Finally, we provide
an isomorphism $H^1(f^* T_{\pp^r}(-\Gamma)) \simeq H^1(f^* T_{\pp^r})$,
which reduces the general case to the case $\Gamma = \emptyset$.
For this, we use the exact sequence
\[0 \to f^* T_{\pp^r}(-\Gamma) \to f^* T_{\pp^r} \to f^* T_{\pp^r}|_\Gamma \simeq T_{\pp^r}|_{f(\Gamma)} \to 0.\]
The composition $H^0(T_{\pp^r}) \to H^0(f^* T_{\pp^r}) \to H^0(T_{\pp^r}|_{f(\Gamma)})$ is surjective (because $f(\Gamma)$
is a collection of linearly independent points);
consequently the restriction map
$H^0(f^* T_{\pp^r}) \to H^0(T_{\pp^r}|_{f(\Gamma)})$ is surjective.
Moreover $T_{\pp^r}|_{f(\Gamma)}$ is punctual, so $H^1(T_{\pp^r}|_{f(\Gamma)}) = 0$.
The long exact sequence in cohomology for the above sequence thus gives the desired isomorphism.
\end{proof}

\begin{lm} \label{tglue}
Let $f \colon C_1 \cup_\Gamma C_2 \to \pp^r$ be a reducible curve,
with $H^1(f|_{C_1}^* T_{\pp^r} (-\Gamma)) = 0$
and $H^1(f|_{C_2}^* T_{\pp^r}) = 0$.
Then $f$ is an interior WBN-curve.
\end{lm}
\begin{proof}
Our assumptions imply, via the exact sequence
\[0 \to f|_{C_1}^* T_{\pp^r}(-\Gamma) \to f^* T_{\pp^r} \to f|_{C_2}^* T_{\pp^r} \to 0,\]
that $H^1(f^* T_{\pp^r}) = 0$; consequently, $f$
is an interior WBN-curve by Lemma~\ref{wbn}.
\end{proof}

\begin{proof}[Proof of Theorem~\ref{main-sp}.]
As mentioned in the introduction, we suppose $C_i$
is general in some component of $\bar{M}_{g_i}(\pp^r, d_i)$,
for both $i \in \{1, 2\}$; and also that
$(r + 1) d_1 - r g_1 + r \geq rn$. Note that
$H^1(f_2^* T_{\pp^r}) = 0$ by Lemma~\ref{wbn}.
By Lemma~\ref{tglue}, it therefore suffices to show
$H^1(f_1^* T_{\pp^r}(-\Gamma)) = 0$.

If $C_1$ is degenerate, then $\Gamma$ is linearly independent
(it is a general set of points which does not span $\pp^r$);
the result thus follows from Lemma~\ref{wbn}.
If $C_1$ is nondegenerate, the result follows from
Theorem~1.2 of \cite{tan}.
\end{proof}

\begin{proof}[Proof of Theorem~\ref{main-hyp}.]
As mentioned in the introduction, we suppose $C$
is general in some component of $\bar{M}_{g'}(\pp^r, d')$,
and $D$ is general in some component of $\bar{M}_{g''}(H, d'')$
(for $d''$ and $g''$ the degree and genus of $D$).
Note that
$H^1(g^* T_{\pp^r}) = 0$ by Lemma~\ref{wbn};
by Lemma~\ref{tglue}, it therefore suffices to show
$H^1(f^* T_{\pp^r}(-\Gamma)) = 0$.

As in the proof of Theorem~\ref{main-sp}, this follows
from Lemma~\ref{wbn} when $f$ is degenerate,
and from Theorem~1.4 of~\cite{tan} when $f$ is nondegenerate.
\end{proof}

\section{WBN-curves and Deformation Theory \label{sec:def}}

In this section, we prove the key lemmas which enable us to use
degeneration in the proof of our theorems.

We begin with a gentle reminder on the deformation theory of maps (c.f.\ Section~3.4 of~\cite{sernesi}): If
$f \colon X \to Y$ is an unramified morphism between lci schemes,
then first-order deformations of $f$ and obstructions to lifting them lie
in the cohomology groups $H^0(N_f)$ and $H^1(N_f)$
respectively of the \emph{normal bundle}:
\[N_f := \hom(N_f^\vee, \oo_X),\]
where $N_f^\vee$ is the \emph{conormal bundle}:
\[\ker(f^* \Omega_Y \to \Omega_X).\]

When $f$ is no longer unramified, the same holds provided we
work in the derived category $\db(\coh(X))$
(for the case of $X$ a curve, which is the only case we shall use, c.f.\ Section~2 of~\cite{ghs}):
Namely, we define the \emph{normal complex}
\[\mathcal{N}_f := \underline{\rhom}(\mathcal{N}_f^\vee, \oo_X),\]
where $\mathcal{N}_f^\vee$ denotes the \emph{conormal complex}:
\[f^* \Omega_Y \to \Omega_X\]
(with $f^* \Omega_Y$ in degree $0$ and $\Omega_X$ in degree $1$).
First-order deformations of $f$ and obstructions to lifting them then lie
in the hypercohomology groups $\mathbb{H}^0(\mathcal{N}_f)$ and $\mathbb{H}^1(\mathcal{N}_f)$.

The reader unaccustomed to derived categories
is advised to assume all maps are unramified in this section --- in which
case the same proofs work with ``normal complex'' replaced by ``normal bundle'' et cetra --- and
then take the result when the necessary maps may not be unramified on faith.

\begin{lm} \label{lm:inter}
Let $S \subset \pp^r$ be a hypersurface, and $\Gamma \subset S$
and $\Delta \subset \pp^r$ be general (possibly empty) sets of points.
Take $f \colon C \to \pp^r$ to be
general in some component of the space of WBN-curves
which are transverse to $S$ and pass through
$\Gamma \cup \Delta$.
Then $\mathbb{H}^1(\mathcal{N}_f(-\Gamma - \Delta)) = 0$
(where by abuse of notation we write $\Gamma \subset C$
and $\Delta \subset C$ for sets of points
mapping injectively under $f$ onto $\Gamma$ and $\Delta$ respectively).

In particular, when the general such $f$ is unramified
(which is the case when the dimension of the linear span of the image of $f$
is at least $2$; c.f.\ Theorem~2 of~\cite{bnc}),
then $H^1(N_f(-\Gamma-\Delta)) = 0$.
\end{lm}
\begin{proof}
Write $n = \# \Gamma$ and $m = \# \Delta$.
The moduli space of such triples $(f, \Gamma, \Delta)$
is then an \'etale cover of the component of
$\bar{M}_{g, m}(\pp^r, d)$ containing $(f, \Delta)$.
It is thus generically reduced of the expected dimension by Lemma~\ref{exp-dim}.
By assumption, this component dominates $S^n \times (\pp^r)^m$,
so the vertical tangent space $\mathbb{H}^0(\mathcal{N}_f(-\Gamma-\Delta))$
is of the expected dimension $\chi(\mathcal{N}_f(-\Gamma-\Delta))$.

By examination, the only nonvanishing cohomology groups of $\mathcal{N}_f^\vee$
are $H^0(\mathcal{N}_f^\vee)$ and $H^1(\mathcal{N}_f^\vee)$.
Moreover, since $f$ is a general WBN-curve, $f$ is generically unramified;
this implies $H^1(\mathcal{N}_f^\vee)$ is punctual.
Using the spectral sequence
\[\ext^i(H^j(\mathcal{N}_f^\vee), \oo(-\Gamma-\Delta)) \Rightarrow \mathbb{H}^{i - j}(\mathcal{N}_f(-\Gamma-\Delta)),\]
we conclude
$\mathbb{H}^k(\mathcal{N}_f(-\Gamma-\Delta)) = 0$ for $k \notin \{0, 1\}$.
Since $\dim \mathbb{H}^0(\mathcal{N}_f(-\Gamma-\Delta)) = \chi(\mathcal{N}_f(-\Gamma-\Delta))$
from above, this implies
$\mathbb{H}^1(\mathcal{N}_f(-\Gamma-\Delta)) = 0$
as desired.
\end{proof}

\begin{lm} \label{union-res}
Let $f \colon X \cup_\Gamma Y \to \pp^r$ be a reducible curve,
and $\mathcal{F}$ be a flat sheaf on $X$. Let
$\Delta \subseteq \Gamma$ be any subset, and write
$g \colon X \cup_\Delta Y \to \pp^r$.
If $\mathbb{H}^1(\mathcal{N}_g|_X \otimes \mathcal{F}) = 0$, then
$\mathbb{H}^1(\mathcal{N}_f|_X \otimes \mathcal{F}) = 0$.
In particular, taking $\Delta = \emptyset$,
if $\mathbb{H}^1(\mathcal{N}_{(f|_X)} \otimes \mathcal{F}) = 0$, then
$\mathbb{H}^1((\mathcal{N}_f)|_X \otimes \mathcal{F}) = 0$.
\end{lm}
\begin{proof}
Write $\bar{\Delta} = \Gamma \smallsetminus \Delta$.
The following diagram with exact rows
\[\begin{CD}
0 @>>> 0 @>>> f^* \Omega_{\pp^r}|_X @= g^* \Omega_{\pp^r}|_X @>>> 0 \\
@. @VVV @VVV @VVV @. \\
0 @>>> \oo_{\bar{\Delta}} @>>> \Omega_{X \cup_\Gamma Y}|_X @>>> \Omega_{X \cup_\Delta Y}|_X @>>> 0
\end{CD}\]
gives an exact triangle in $\db(\coh(X))$:
\[\oo_{\bar{\Delta}}[1] \to \mathcal{N}_f^\vee|_X \to \mathcal{N}_{g}^\vee|_X \to.\]
Upon applying $\underline{\rhom}(-, \oo_X)$, twisting by $\mathcal{F}$,
and taking hypercohomology, we get a long exact sequence
\[ \cdots \to \mathbb{H}^1(\mathcal{N}_g|_X \otimes \mathcal{F}) \to \mathbb{H}^1(\mathcal{N}_f|_X \otimes \mathcal{F}) \to \mathbb{H}^1(\underline{\rhom}(\oo_{\bar{\Delta}}[1], \oo_X) \otimes \mathcal{F}) \to \cdots.\]
It thus remains to show
$\mathbb{H}^1(\underline{\rhom}(\oo_{\bar{\Delta}}[1], \oo_X) \otimes \mathcal{F}) = 0$.
But
\[\mathbb{H}^1(\underline{\rhom}(\oo_{\bar{\Delta}}[1], \oo_X) \otimes \mathcal{F}) \simeq \mathbb{H}^1(\underline{\rhom}(\oo_{\bar{\Delta}}[1], \mathcal{F})) \simeq \ext^2(\oo_{\bar{\Delta}}, \mathcal{F}) = 0. \qedhere\]
\end{proof}

\begin{lm} \label{lm:union-smooth} Let
$f \colon X \cup_\Gamma Y \to \pp^r$ be a reducible curve, and $\mathcal{F}$ be a flat sheaf on $X \cup_\Gamma Y$.
If $\mathbb{H}^1(\mathcal{N}_f(-\Gamma) \otimes \mathcal{F}|_X) = \mathbb{H}^1(\mathcal{N}_f \otimes \mathcal{F}|_Y) = 0$,
then $\mathbb{H}^1(\mathcal{N}_f \otimes \mathcal{F}) = 0$.
In particular, taking $\mathcal{F} = \oo$,
if $\mathbb{H}^1(\mathcal{N}_f(-\Gamma)|_X) = \mathbb{H}^1(\mathcal{N}_f|_Y) = 0$,
then $\mathbb{H}^1(\mathcal{N}_f) = 0$,
and so $f$ is an interior curve.
\end{lm}
\begin{proof}
Note that a smooth point of a scheme lies in a unique component;
thus $\mathbb{H}^1(\mathcal{N}_f) = 0$ implies $f$ is an interior curve.
To show $\mathbb{H}^1(\mathcal{N}_f \otimes \mathcal{F}) = 0$ as desired, we use
(the long exact sequence in hypercohomology attached to) the normal complex exact triangle
\[\mathcal{N}_f|_X(-\Gamma) \otimes \mathcal{F} \to \mathcal{N}_f \otimes \mathcal{F} \to \mathcal{N}_f \otimes \mathcal{F} |_Y \to,\]
which, as desired, reduces $\mathbb{H}^1(\mathcal{N}_f \otimes \mathcal{F}) = 0$ to
\[\mathbb{H}^1(\mathcal{N}_f \otimes \mathcal{F}|_X(-\Gamma)) = \mathbb{H}^1(\mathcal{N}_f \otimes \mathcal{F}|_Y) = 0. \qedhere\]
\end{proof}


\begin{lm} \label{can-degenerate}
Fix $i \in \{1, 2\}$, and
suppose $f^\circ_i \colon C^\circ_i \to \pp^r$ is a specialization of $C_i \to \pp^r$ in Theorem~\ref{main}, which
still passes through the general set of points $\Gamma$, and satisfies
$\mathbb{H}^1(\mathcal{N}_{f^\circ_i}) = 0$.
Then it suffices to show the resulting curve is a BN-curve after we
specialize $C_i$ to $C^\circ_i$ in Theorem~\ref{main}.
(So long as we leave the other curve general.)
\end{lm}
\begin{proof}
Without loss of generality, we suppose $i = 1$.
Write $\mathcal{K}_j = \bar{M}_{g_j, n}(\pp^r, d_j)$
and $\mathcal{P} = (\pp^r)^n$.

Our pair of curves $(C_1, C_2)$ in Theorem~\ref{main} then corresponds to
a point (which we may as well suppose is the generic point),
in some component of the fiber product $\mathcal{K}_1 \times_{\mathcal{P}} \mathcal{K}_2$
which dominates $\mathcal{P}$.
As $\mathcal{K}_1$ is irreducible, any component of
$\mathcal{K}_1 \times_{\mathcal{P}} \mathcal{K}_2$ which dominates $\mathcal{P}$
in fact dominates $\mathcal{K}_1$. This means there exists a specialization of $(C_1, C_2)$
of the form $(C_1^\circ, \bar{C}_2)$, where $\bar{C}_2$ satisfies the assumptions
(and therefore conclusions) of Lemma~\ref{lm:inter}.

Applying Lemma~\ref{union-res} and~\ref{lm:union-smooth}, the specialization
$C_1^\circ \cup_\Gamma \bar{C}_2 \to \pp^r$ is an interior curve.
If it is a BN-curve, we can thus conclude $C_1 \cup_\Gamma C_2 \to \pp^r$ is an interior BN-curve
as desired.
\end{proof}

\begin{lm} \label{can-degenerate-small-hyp-g}
Suppose $g^\circ \colon D^\circ \to H$ is a specialization of $g$ in Theorem~\ref{small-hyp}, which
still passes through the general set of points $\Gamma \subset H$, and satisfies
$\mathbb{H}^1(\mathcal{N}_{g^\circ}) = H^1((g^\circ)^* \oo_{\pp^r}(1)(\Gamma)) = 0$.
Then it suffices to show the resulting curve is a BN-curve after we
specialize $g$ to $g^\circ$ in Theorem~\ref{small-hyp}.
(So long as we leave $f$ general.)
\end{lm}
\begin{proof}
Writing $h \colon C \cup_\Gamma D^\circ \to \pp^r$ for the resulting curve,
an analogous argument as in Lemma~\ref{can-degenerate} works here so long as
$\mathbb{H}^1(\mathcal{N}_h|_{D^\circ}) = 0$.
To see this, we note that since $f$ is general, $f$ is transverse to $H$; we therefore have the exact triangle
\[\mathcal{N}_{g^\circ} \to \mathcal{N}_h|_{D^\circ} \to (g^\circ)^* N_{H/\pp^r}(\Gamma)[0] \to,\]
which gives rise to the long exact sequence
\[\cdots \to \mathbb{H}^1(\mathcal{N}_{g^\circ}) \to \mathbb{H}^1(\mathcal{N}_h|_{D^\circ}) \to H^1((g^\circ)^* \oo_{\pp^r}(1)(\Gamma)) \to \cdots.\]
Our assumption that 
$\mathbb{H}^1(\mathcal{N}_{g^\circ}) = H^1((g^\circ)^* \oo_{\pp^r}(1)(\Gamma)) = 0$
then implies that $\mathbb{H}^1(\mathcal{N}_h|_{D^\circ}) = 0$ as desired.
\end{proof}

\begin{lm} \label{can-degenerate-small-hyp-f}
Suppose $f^\circ \colon C^\circ \to H$ is a specialization of $f$ in Theorem~\ref{small-hyp}, which
still passes through the general set of points $\Gamma \subset H$, is transverse to $H$ along $\Gamma$, and satisfies
\mbox{$\mathbb{H}^1(\mathcal{N}_{f^\circ}) = 0$}.
Then it suffices to show the resulting curve is a BN-curve after we
specialize $f$ to $f^\circ$ in Theorem~\ref{small-hyp}, provided that $d'' \geq g'' + r - 1$.
(So long as we leave $g$ general.)
\end{lm}
\begin{proof}
Writing $h \colon C^\circ \cup_\Gamma D \to \pp^r$ for the resulting curve,
an analogous argument as in Lemma~\ref{can-degenerate} works here so long as
$\mathbb{H}^1(\mathcal{N}_h|_D(-\Gamma)) = 0$.
To see this, we note that since $f^\circ$ is transverse to $H$ along $\Gamma$ by assumption,
we have an exact triangle
\[\mathcal{N}_{g}(-\Gamma) \to \mathcal{N}_h|_D(-\Gamma) \to g^* N_{H/\pp^r}[0] \to,\]
which gives rise to the long exact sequence
\[\cdots \to \mathbb{H}^1(\mathcal{N}_{g}(-\Gamma)) \to \mathbb{H}^1(\mathcal{N}_h|_D(-\Gamma)) \to H^1(g^* \oo_{\pp^r}(1)) \to \cdots.\]
Since $g \colon D \to H$ is general, our assumption that $d'' \geq g'' + r - 1$
implies $H^1(g^* \oo_{\pp^r}(1)) = 0$.
As $\mathbb{H}^1(\mathcal{N}_{g}(-\Gamma)) = 0$ by Lemma~\ref{lm:inter},
we conclude
$\mathbb{H}^1(\mathcal{N}_h|_D(-\Gamma)) = 0$ as desired.
\end{proof}

\begin{lm} \label{can-degenerate-small-hyp-f-var}
Suppose $f^\circ \colon C^\circ \to H$ is a specialization of $f$ in Theorem~\ref{small-hyp}, which
still passes through the general set of points $\Gamma \subset H$, is transverse to $H$ along a subset $\Gamma' \subset \Gamma$, and satisfies
$\mathbb{H}^1(\mathcal{N}_{f^\circ}(-\Gamma)) = 0$.
Then it suffices to show the resulting curve is a BN-curve after we
specialize $f$ to $f^\circ$ in Theorem~\ref{small-hyp}, provided that $d'' \geq g'' + r - 1 - \# \Gamma'$.
(So long as we leave $g$ general.)
\end{lm}
\begin{proof}
Writing $h \colon C^\circ \cup_\Gamma D \to \pp^r$ for the resulting curve,
an analogous argument as in Lemma~\ref{can-degenerate} works here so long as
$\mathbb{H}^1(\mathcal{N}_h|_D) = 0$.
Write $h' \colon C^\circ \cup_{\Gamma'} D \to \pp^r$;
by Lemma~\ref{union-res}, it suffices to show $\mathbb{H}^1(\mathcal{N}_{h'}|_D) = 0$.

To see this, we note that since $f^\circ$ is transverse to $H$ along $\Gamma'$ by assumption,
we have an exact triangle
\[\mathcal{N}_{g} \to \mathcal{N}_{h'}|_D \to g^* N_{H/\pp^r}(\Gamma')[0] \to,\]
which gives rise to the long exact sequence
\[\cdots \to \mathbb{H}^1(\mathcal{N}_{g}) \to \mathbb{H}^1(\mathcal{N}_h|_D) \to H^1(g^* \oo_{\pp^r}(1)(\Gamma')) \to \cdots.\]
Since $g \colon D \to H$ is general, $d'' \geq g'' + r - 1 - \# \Gamma'$
implies $\dim H^1(g^* \oo_{\pp^r}(1)) \leq \# \Gamma'$,
and thus $H^1(g^* \oo_{\pp^r}(1)(\Gamma')) = 0$.
As $\mathbb{H}^1(\mathcal{N}_{g}) = 0$ by Lemma~\ref{lm:inter},
$\mathbb{H}^1(\mathcal{N}_{h'}|_D) = 0$ as desired.
\end{proof}

\begin{lm} \label{cohom-small-mid}
If $\hat{f}$ is as in Theorem~\ref{small-mid}, then
$\mathbb{H}^1(\mathcal{N}_{\hat{f}}|_{\pp^1}(-\Gamma)) = 0$
\end{lm}
\begin{proof}
To see that $\mathbb{H}^1(\mathcal{N}_{\hat{f}}|_{\pp^1}(-\Gamma)) = 0$,
let $\Delta \subseteq \Gamma$ be a subset of size $r - a \leq a + 2 = \# \Gamma$
such that the image of $2 \Delta \subset C$ under $\hat{f}$ along with
$\hat{f}(\pp^1)$ spans $\pp^r$,
and write $g \colon C \cup_\Delta \pp^1 \to \pp^r$.
By Lemma~\ref{union-res}, it suffices to show
$\mathbb{H}^1(\mathcal{N}_g|_{\pp^1}(-\Gamma)) = 0$.
By construction, $g$ is unramified in a neighborhood of $\pp^1$
in $C \cup_\Delta \pp^1$, and so the normal complex
$\mathcal{N}_g$ can be identified with the
normal sheaf $N_g$.

Write $\Lambda$ for the linear span of $g(\pp^1)$, and for $p \in \Delta$,
let $H_p$ be the hyperplane spanned by $\Lambda$ and the image of
$2(\Delta \smallsetminus \{p\}) \subset C$ under $g$.
Then $\Lambda$ is the complete intersection of the $H_p$, and so
we obtain an exact sequence for the normal bundle of the image $g(\pp^1)$:
\[0 \to N_{g(\pp^1) / \Lambda} \to N_{g(\pp^1)} \to \bigoplus_{p \in \Delta} N_{H_p}|_{g(\pp^1)} \to 0.\]
Applying Corollary~3.2 of~\cite{hh} (stated for $r = 3$
when $g$ is an immersion, but the proof given applies for $r$ arbitrary
and a long as $g$ is unramified in a neighborhood of the given component),
this induces an exact sequence
\[0 \to N_{g(\pp^1) / \Lambda} \to N_g|_{\pp^1} \to \bigoplus_{p \in \Delta} N_{H_p}|_{g(\pp^1)}(p) \to 0.\]
Twisting by $-\Gamma$, it remains to show (for $p \in \Delta$):
\[H^1(N_{g(\pp^1) / \Lambda}(-\Gamma)) = H^1(N_{H_p}|_{g(\pp^1)}(p)(-\Gamma)) = 0\]
But $H^1(N_{g(\pp^1) / \Lambda}(-\Gamma))$
vanishes by Theorem~1.3 of~\cite{aly},
and
\[N_{H_p}|_{g(\pp^1)}(p)(-\Gamma) \simeq g^* \oo_{H_p}(1)(p)(-\Gamma) \simeq \oo_{\pp^1}(a + 1 - (a + 2)) = \oo_{\pp^1}(-1),\]
which has vanishing $H^1$ as desired.
\end{proof}

\begin{lm} \label{can-degenerate-small-mid}
Suppose $f^\circ \colon C^\circ \to \pp^r$ is a specialization of $f$ in Theorem~\ref{small-mid}, which
satisfies $\mathbb{H}^1(\mathcal{N}_{f^\circ}) = 0$.
Then it suffices to show the resulting curve is a BN-curve after we
specialize $f$ to $f^\circ$ in Theorem~\ref{small-mid}.
\end{lm}
\begin{proof}
By Lemma~\ref{cohom-small-mid}, we have
$\mathbb{H}^1(\mathcal{N}_{\hat{f^\circ}}|_{\pp^1}(-\Gamma)) = 0$.
This implies that any deformation of $f^\circ$
lifts to a deformation of $\hat{f^\circ}$;
moreover since $\mathbb{H}^1(\mathcal{N}_{f^\circ}) = 0$ by assumption,
Lemma~\ref{lm:union-smooth} implies
$\mathbb{H}^1(\mathcal{N}_{\hat{f^\circ}}) = 0$ and so $\hat{f^\circ}$ is an interior curve.
\end{proof}

\section{Proof of Theorem~\ref{main} \label{sec:main}}

For this section, we adopt the notation of Theorem~\ref{main}; that is, we let
$C_i \to \pp^r$ (for $i \in \{1, 2\}$) be nonspecial BN-curves
of degree $d_i$ and genus $g_i$, which pass through
a set $\Gamma \subset \pp^r$ of $n$ general points.

Our argument will be by induction on the total degree $d_1 + d_2$,
and for fixed total degree by induction on $n$.
There will be several cases to consider, but in each case our argument will
follow the following outline:

As mentioned in the introduction,
we may begin by supposing that both curves $C_i \to \pp^r$
are general in some component of the space of NNS-curves
passing through $\Gamma$;
our goal is to degenerate one curve, say $f_1 \colon C_1 \to \pp^r$,
to a reducible curve $f_1^\circ \colon C_1' \cup_{\Gamma_0} C_1'' \to \pp^r$,
where $C_1' \to \pp^r$ and $C_1'' \to \pp^r$ are NNS-curves and NS-curves respectively,
with specified degrees $d_1'$ and $d_1''$ (with $d_1' + d_1'' = d_1$),
with specified genera $g_1'$ and $g_1''$, and meeting eachother in a specified number of points $n_0 = \#\Gamma_0$
(with $g_1' + g_1'' + n_0 - 1 = g_1$).
Let $n'$ and $n''$ be the integers with $n' + n'' = n$,
for which we desire $C_1'$ to pass through $n'$ points $\Gamma' \subseteq \Gamma$
and $C_1''$ to pass through $n''$ points $\Gamma'' \subseteq \Gamma$, with $\Gamma' \cup \Gamma'' = \Gamma$.
We verify:
\begin{enumerate}
\item \label{C1p-inequalities}
We have $d_1' \geq g_1' + r$. Also,
$(r + 1) d_1' - (r - 3)(g_1' - 1) - (r - 1) (n' + n_0) \geq 0$,
with strict inequality in the cases
$(d_1', g_1', r) \in \{(5, 2, 3), (7, 2, 5)\}$.

\item \label{C1pp-inequalities}
\begin{enumerate}
\item \label{C1pp-inequalities-nd}
If $d_1'' \geq g_1'' + r$, then
$(r + 1) d_1'' - (r - 3)(g_1'' - 1) - (r - 1) \max(n'', n_0) \geq 0$,
with strict inequality in the cases
$(d_1'', g_1'', r) \in \{(5, 2, 3), (7, 2, 5)\}$.
\item \label{C1pp-inequalities-de}
Otherwise, $d_1'' + 1 - g_1'' - \max(n'', n_0) \geq 0$.
Note that when $C_1''$ is a line (i.e.\ $d_1'' = 1$ and $g_1'' = 0$),
this inequality becomes $\max(n'', n_0) \leq 2$.
\end{enumerate}

\item \label{for-main-sp}
$C_1' \cup_{\Gamma_0} C_1'' \to \pp^r$ satisfies the assumptions of
Theorem~\ref{main-sp}.
In every application, we can just check $n_0 \leq r + 2$,
since this implies
\[(r + 1) d_1' - r g_1' + r \geq (r + 1)(g_1' + r) - r g_1' + r = r(r + 2) + g_1' \geq r(r + 2) \geq rn_0.\]
\end{enumerate}

These assumptions imply we can degenerate $C_1 \to \pp^r$
to
$C_1' \cup_{\Gamma_0} C_1'' \to \pp^r$: Conditions~\ref{C1p-inequalities}
and~\ref{C1pp-inequalities} imply --- via Theorem~\ref{cor14} (for \ref{C1p-inequalities}/\ref{C1pp-inequalities-nd}),
or Remark~\ref{aly-deg} (for \ref{C1pp-inequalities-de}) ---
that $C_1'$ and $C_1''$
pass through a set $\Gamma_0$ of $n_0$ general points.
Applying Theorem~\ref{main-sp} and Condition~\ref{for-main-sp}, this shows
$C_1' \cup_{\Gamma_0} C_1'' \to \pp^r$ is a BN-curve.
Additionally, our assumptions imply
such a degeneration can be performed so that
$C_1'$ and $C_1''$ pass through sets $\Gamma'$ and $\Gamma''$,
of cardinality $n'$ and $n''$, with
$\Gamma = \Gamma' \cup \Gamma''$:
In fact, the construction can be phrased as
first finding a curve $C_1''$ through $\Gamma''$; such a curve
also passes through a general set $\Gamma_0$ of $n_0$ general points by assumption;
we then find a curve $C_1'$ passing through $\Gamma' \cup \Gamma_0$.
Moreover, applying Lemmas~\ref{lm:inter}, \ref{union-res}, and~\ref{lm:union-smooth},
and using the generality of $\Gamma_0$, we see that
$\mathbb{H}^1(\mathcal{N}_{f_1^\circ}) = 0$.
In order to prove Theorem~\ref{main},
from Lemma~\ref{can-degenerate}, it therefore suffices to show
$(C_1' \cup_{\Gamma_0} C_1'') \cup_\Gamma C_2 \to \pp^r$
is a BN-curve.

We then verify:

\begin{enumerate}[resume]
\item \label{induct-inner} $C_1'' \cup_{\Gamma''} C_2 \to \pp^r$ satisfies the assumptions of either
\begin{enumerate}
\item \label{induct-inner-main}
Theorem~\ref{main}.
In every application, we can just check $d_2 = g_2 + r$ and $d_1'' = g_1'' + r$.
\item \label{induct-inner-main-sp}
Theorem~\ref{main-sp}.
Note that this is automatic if Condition~\ref{C1pp-inequalities-de} was satisfied: 
\[(r + 1)d_1'' - rg_1'' + r \geq r(d_1'' - g_1'' + 1) \geq r \cdot \max(n'', n_0) \geq rn''.\]
\end{enumerate}
\end{enumerate}

This implies $C_1'' \cup_{\Gamma''} C_2 \to \pp^r$ is a BN-curve,
by application of Theorem~\ref{main-sp} (for \ref{induct-inner-main-sp}),
or by induction as necessary (for \ref{induct-inner-main});
in the second case,
note that the total degree $d_1'' + d_2$ satisfies $d_1'' + d_2 < d_1 + d_2$.

We then verify:

\begin{enumerate}[resume]
\item \label{induct-outer} $C_1' \cup_{\Gamma_0 \cup \Gamma'} (C_1'' \cup_{\Gamma''} C_2) \to \pp^r$
satisfies the assumptions of either
\begin{enumerate}
\item \label{induct-outer-main}
Theorem~\ref{main}, in which case we also have $n_0 \leq n''$.
In every application, we may verify the assumptions of
Theorem~\ref{main} by either checking
that both curves are limit linearly normal, i.e.\ checking that
\[d_1' = g_1' + r \tand d_1'' + d_2 = g_1'' + g_2 + n'' - 1 + r;\]
or by taking $i = 2$, for which it suffices to show
\[(r + 1) (d_1'' + d_2) - (r - 3)((g_1'' + g_2 + n'' - 1) - 1) - (r - 1)(n_0 + n') \geq 4. \]

\item \label{induct-outer-main-sp}
Theorem~\ref{main-sp}. 
In every application, we may check $n_0 + n' \leq r + 2$,
which implies 
the required inequality $(r + 1) d_1' - r g_1' + r \geq r(n_0 + n')$
as in Condition~\ref{for-main-sp}.
\end{enumerate}
\end{enumerate}

This either completes the proof (for \ref{induct-outer-main-sp}),
or reduces us inductively to another
instance of Theorem~\ref{main} (for \ref{induct-outer-main}).
This other instance has the same total degree;
and since $n_0 \leq n''$ implies $\#(\Gamma_0 \cup \Gamma') = n_0 + n' \leq n'' + n' = n$,
the value of $n$ does not increase. If $n_0 < n''$ in \ref{induct-outer-main},
then $n$ decreases, so we are done by induction;
otherwise, we must show this inequality is at least strict
eventually after running
the entire argument a finite number of times.

Before starting the proof, let us rewrite our assumption that
$C_1 \cup_\Gamma C_2 \to \pp^r$ has nonnegative Brill--Noether number in a more convenient form:
\begin{align*}
\rho(C_1 \cup_\Gamma C_2 \to \pp^r) &= (r + 1)(d_1 + d_2) - r(g_1 + g_2 + n - 1) - r(r + 1) \\
&= [(r + 1) d_1 + rg_1 - r(r + 1)] + [(r + 1) d_1 + rg_1 - r(r + 1)] - rn + r(r + 2) \\
&= \rho(C_1 \to \pp^r) + \rho(C_2 \to \pp^r) - rn + r(r + 2);
\end{align*}
and so our assumption is equivalent to
\begin{equation} \label{n-bound-1}
n \leq r + 2 + \frac{\rho(C_1 \to \pp^r) + \rho(C_2 \to \pp^r)}{r}.
\end{equation}

\begin{proof}[\hypertarget{nr2}{Proof when $n \leq r + 2$.}]
As in Condition~\ref{for-main-sp}, this implies the assumptions of
Theorem~\ref{main-sp} are satisfied; applying
Theorem~\ref{main-sp} thus yields the desired result.
\end{proof}

\begin{proof}[\hypertarget{both-linnorm}{Proof when $d_i = g_i + r$ for both $i \in \{1, 2\}$.}]
For fixed $d_1 + d_2$ and fixed $n$, we argue by induction on $\min(g_1, g_2)$.
Without loss of generality, suppose $g_1 \leq g_2$. Since $\rho(g + r, g, r) = g$,
Equation~\eqref{n-bound-1} becomes
\begin{equation} \label{n-bound}
n \leq r + 2 + \frac{g_1 + g_2}{r}.
\end{equation}
We now consider several cases:

\paragraph{\boldmath If
$g_1 \geq 1$, and if $(r - 1)n \leq 4g_1 + r^2 + 2r - 7$
with strict inequality in the cases
$(g_1, r) \in \{(3, 3), (3, 5)\}$:}
First note that, since $d_1 = g_1 + r$, our second inequality rearranges to give
\[(r - 1)n \leq (r + 1)(d_1 - 1) - (r - 3)((g_1 - 1) - 1),\]
with strict inequality in the cases
$(d_1 - 1, g_1 - 1, r) \in \{(5, 2, 3), (7, 2, 5)\}$.

We degenerate $C_1 \to \pp^r$ to $C_1' \cup_{\Gamma_0} C_1'' \to \pp^r$,
where $C_1'$ is of degree $d_1' = d_1 - 1$ and genus $g_1' = g_1 - 1$,
and $C_1''$ is a line ($d_1'' = 1$ and $g_1'' = 0$), and $n_0 = \# \Gamma_0 = 2$;
we let $n' = n - 2$ and $n'' = 2$. Conditions
\ref{C1p-inequalities}, \ref{C1pp-inequalities-de}, \ref{for-main-sp},
\ref{induct-inner-main-sp}, and \ref{induct-outer-main} (with both curves limit linearly normal),
from the above discussion are easily verified.
The inequality $n_0 \leq n''$ in \ref{induct-outer-main} is an equality;
but $\min(g_1, g_2)$ decreases, so that completes the induction.

\paragraph{\boldmath If 
$g_2 \geq r$, and if $(r - 1)n \leq 4g_2 + r^2 - r - 4$ with strict inequality in the cases
$(g_2, r) \in \{(5, 3), (7, 5)\}$:}
Exchanging indices, we can instead consider the case when
$g_1 \geq r$, and $(r - 1)n \leq 4g_1 + r^2 - r - 4$ with strict inequality in the cases
$(g_1, r) \in \{(5, 3), (7, 5)\}$.
Since $d_1 = g_1 + r$, our second inequality rearranges to give
\[(r - 1)(n - 1) \leq (r + 1)(d_1 - r) - (r - 3)((g_1 - r) - 1),\]
with strict inequality in the cases
$(d_1 - r, g_1 - r, r) \in \{(5, 2, 3), (7, 2, 5)\}$.

We degenerate $C_1 \to \pp^r$ to $C_1' \cup_{\Gamma_0} C_1'' \to \pp^r$,
where $C_1'$ is a rational normal curve ($d_1' = r$ and $g_1' = 0$),
and $C_1''$ is of degree $d_1' = d_1 - r$ and genus $g_1' = g_1 - r$,
and $n_0 = \# \Gamma_0 = r + 1$;
we let $n' = 1$ and $n'' = n - 1$.
Conditions
\ref{C1p-inequalities}, \ref{C1pp-inequalities-nd}, \ref{for-main-sp},
\ref{induct-inner-main}, and \ref{induct-outer-main-sp},
from the above discussion are easily verified,
so that completes the induction.

\hypertarget{g1z-g2r}{\paragraph{\boldmath If $g_1 = 0$ and $g_2 \leq r - 1$:}} Equation~\eqref{n-bound}
implies $n \leq r + 3 - \frac{1}{r}$; as $n$ is an integer, $n \leq r + 2$,
so this falls into a case already considered (\hyperlink{nr2}{``Proof when $n \leq r + 2$''}).

\hypertarget{g1z-g2i}{\paragraph{\boldmath If $g_1 = 0$, and $(r - 1)n \geq 4g_2 + r^2 - r - 4$ with strict inequality except when $(g_2, r) \in \{(5, 3), (7, 5)\}$:}}
Since $C_1$ is a rational normal curve, it can only pass through $n$
points if
\[(r - 1)n \leq (r + 1) d_1 - (r - 3)(g_1 - 1) = (r - 1)(r + 3) \quad \Leftrightarrow \quad n \leq r + 3.\]
Our inequality then gives
\[4g_2 + r^2 - r - 4 \leq (r - 1)(r + 3) \quad \Leftrightarrow \quad g_2 \leq \frac{3r + 1}{4},\]
with strict inequality unless $(g_2, r) \in \{(5, 3), (7, 5)\}$.
Thus $g_2 \leq r - 1$,
so this falls into a case already considered (\hyperlink{g1z-g2r}{``If $g_1 = 0$ and $g_2 \leq r - 1$''}).

\hypertarget{g1i-g2r}{\paragraph{\boldmath If $(r - 1)n \geq 4g_1 + r^2 + 2r - 7$
with strict inequality unless
$(g_1, r) \in \{(3, 3), (3, 5)\}$, and $g_2 \leq r - 1$:}}
From Equation~\eqref{n-bound},
\[n \leq r + 2 + \frac{g_1 + g_2}{r} \leq r + 2 + \frac{2g_2}{r} \leq r + 2 + \frac{2(r - 1)}{r} = r + 4 - \frac{2}{r} \quad \Rightarrow \quad n \leq r + 3. \]
Consequently,
\[4g_1 + r^2 + 2r - 7 \leq (r - 1)n \leq (r - 1)(r + 3) \quad \Leftrightarrow \quad g_1 \leq 1,\]
with strict inequality unless
$(g_1, r) \in \{(3, 3), (3, 5)\}$. In particular, $g_1 = 0$,
so this falls into a case already considered (\hyperlink{g1z-g2r}{``If $g_1 = 0$ and $g_2 \leq r - 1$''}). 

\paragraph{\boldmath If $(r - 1)n \geq 4g_1 + r^2 + 2r - 7$
with strict inequality unless
$(g_1, r) \in \{(3, 3), (3, 5)\}$, and $(r - 1)n \geq 4g_2 + r^2 - r - 4$
with strict inequality unless $(g_2, r) \in \{(5, 3), (7, 5)\}$:}
Adding these two inequalities together, we obtain
\[(2r - 2)n \geq 4(g_1 + g_2) + 2r^2 + r - 11,\]
with strict inequality unless $r \in \{3, 5\}$.
Combining with Equation~\eqref{n-bound},
\[(2r - 2) \left(r + 2 + \frac{g_1 + g_2}{r}\right) \geq 4(g_1 + g_2) + 2r^2 + r - 11;\]
or upon rearrangement
\[g_1 + g_2 \leq \frac{r^2 + 7r}{2r + 2} = r + 1 - \frac{(r - 1)(r - 2)}{2r + 2}.\]
This holds with strict inequality unless $r \in \{3, 5\}$; in particular we always have
\[g_1 + g_2 \leq r.\]
Since the case $g_1 = 0$ was already considered
(\hyperlink{g1z-g2i}{``If $g_1 = 0$, and $(r - 1)(n - 1) \geq 4g_2 + r^2 - 2r - 3$ with strict inequality except when $(g_2, r) \in \{(5, 3), (7, 5)\}$''}),
we may suppose $g_1 \geq 1$;
the above inequality then gives $g_2 \leq r - 1$, so we are again in a case already considered
(\hyperlink{g1i-g2r}{``If $(r - 1)n \geq 4g_1 + r^2 + 2r - 7$
with strict inequality unless
$(g_1, r) \in \{(2, 3), (2, 5)\}$, and $g_2 \leq r - 1$''}).
\end{proof}

Except when both curves are limit linearly normal (which was considered
above), we have by assumption that for at least one $i \in \{1, 2\}$,
\begin{equation} \label{assume}
(r + 1) d_i - (r - 3)(g_i - 1) - (r - 1)n \geq \begin{cases} 
2 & \text{if $d_i > g_i + r$;} \\
4 & \text{if $d_i = g_i + r$.}
\end{cases}
\end{equation}
\emph{For the remainder of this section, assume without loss of generality
that this happens for $i = 1$.}

\begin{proof}[\hypertarget{not-linnorm}{Proof when $d_1 > g_1 + r$, assuming Equation~\eqref{assume} is strict if $(d_1, g_1, r) \in \{(6, 2, 3), (8, 2, 5)\}$.}] \ \\ 
First note that Equation~\eqref{assume} rearranges to give
\[(r - 1)(n - 1) \leq (r + 1)(d_1 - 1) - (r - 3)(g_1 - 1),\]
with strict inequality in the cases
$(d_1 - 1, g_1, r) \in \{(5, 2, 3), (7, 2, 5)\}$.

Next note that $(r + 1) d_2 - (r - 3)(g_2 - 1) \geq (r - 1)n$ (c.f.\ Theorem~\ref{cor14});
or upon rearrangement,
\[(r + 1) (d_2 + 1) - (r - 3)((g_2 + 1) - 1) - (r - 1)(n - 1) \geq r + 3 \geq 4.\]

We degenerate $C_1 \to \pp^r$ to $C_1' \cup_{\Gamma_0} C_1'' \to \pp^r$,
where $C_1'$ is of degree $d_1' = d_1 - 1$ and genus $g_1' = g_1$,
and $C_1''$ is a line ($d_1'' = 1$ and $g_1'' = 0$), and $n_0 = \# \Gamma_0 = 1$;
we let $n' = n - 2$ and $n'' = 2$. Conditions
\ref{C1p-inequalities}, \ref{C1pp-inequalities-de}, \ref{for-main-sp},
\ref{induct-inner-main-sp}, and \ref{induct-outer-main} (with $i = 2$),
from the previous discussion
are easily verified.
The inequality $n_0 \leq n''$ in \ref{induct-outer-main} is strict, so that completes the induction.
\end{proof}

\begin{proof}[\hypertarget{one-linnorm}{Proof when $d_1 = g_1 + r$.}]
We may suppose $d_2 > g_2 + r$, since the case where $d_i = g_i + r$ for both $i \in \{1, 2\}$
has been considered already
(\hyperlink{both-linnorm}{``Proof when $d_i = g_i + r$ for both $i \in \{1, 2\}$''}).

As above, note that $(r + 1) d_2 - (r - 3)(g_2 - 1) \geq (r - 1)n$ (c.f.\ Theorem~\ref{cor14});
or upon rearrangement,
\[(r + 1) (d_2 + 1) - (r - 3)((g_2 + 1) - 1) - (r - 1)n \geq 4 > 2.\]

First suppose that Equation~\eqref{assume} is strict if $(d_1, g_1, r) \in \{(6, 3, 3), (8, 3, 5)\}$.
Upon rearrangement, this gives
\[(r - 1)n \leq (r + 1)(d_1 - 1) - (r - 3)((g_1 - 1) - 1),\]
with strict inequality in the cases
$(d_1 - 1, g_1, r) \in \{(5, 2, 3), (7, 2, 5)\}$.

In this case, we degenerate $C_1 \to \pp^r$ to $C_1' \cup_{\Gamma_0} C_1'' \to \pp^r$,
where $C_1'$ is of degree $d_1' = d_1 - 1$ and genus $g_1' = g_1 - 1$,
and $C_1''$ is a line ($d_1'' = 1$ and $g_1'' = 0$), and $n_0 = \# \Gamma_0 = 2$;
we let $n' = n - 2$ and $n'' = 2$. Conditions
\ref{C1p-inequalities}, \ref{C1pp-inequalities-de}, \ref{for-main-sp},
\ref{induct-inner-main-sp}, and \ref{induct-outer-main} (with $i = 2$),
from the previous discussion are easily verified.
The inequality $n_0 \leq n''$ in \ref{induct-outer-main} is an equality;
but upon exchanging indices, we are in the previous case (\hyperlink{not-linnorm}{``Proof when $d_1 > g_1 + r$, assuming Equation~\eqref{assume} is strict if $(d_1, g_1, r) \in \{(6, 2, 3), (8, 2, 5)\}$''}),
so that completes the induction.

It remains to consider the cases where
$(d_1, g_1, r) \in \{(6, 3, 3), (8, 3, 5)\}$
and Equation~\eqref{assume} is an equality, i.e.
\[(r + 1) d_1 - (r - 3)(g_1 - 1) - (r - 1)n = 4 \quad \Rightarrow \quad n = \frac{(r + 1) d_1 - (r - 3)(g_1 - 1) - 4}{r - 1} = 10.\]
From Equation~\eqref{n-bound-1}, that gives
\[10 \leq r + 2 + \frac{3 + \rho(C_2 \to \pp^r)}{r} \quad \Rightarrow \quad \rho(C_2 \to \pp^r) \geq -r^2 + 8r - 3.\]
This gives
\begin{align*}
(r + 1) d_2 - (r - 3)(g_2 - 1) &= [(r + 1) d_2 - r g_2 - r(r + 1)] + r(r + 2) + 3(g_2 - 1) \\
&\geq - r^2 + 8r - 3 + r(r + 2) - 3 \\
&= 10r - 6.
\end{align*}
Consequently,
\[(r + 1) d_2 - (r - 3)(g_2 - 1) - (r - 1)n \geq 10r - 6 - 10(r - 1) = 4 > 2.\]
Exchanging indices, this falls into the previous case
(\hyperlink{not-linnorm}{``Proof when $d_1 > g_1 + r$, assuming Equation~\eqref{assume} is strict if $(d_1, g_1, r) \in \{(6, 2, 3), (8, 2, 5)\}$''}).
\end{proof}

\begin{proof}[Proof when $(d_1, g_1, r) \in \{(6, 2, 3), (8, 2, 5)\}$ and Equation~\eqref{assume} is an equality.]
We have
\[(r + 1) d_1 - (r - 3)(g_1 - 1) - (r - 1)n = 2 \quad \Rightarrow \quad n = \frac{(r + 1) d_1 - (r - 3)(g_1 - 1) - 2}{r - 1} = 11.\]
From Equation~\eqref{n-bound-1}, that gives
\[11 \leq r + 2 + \frac{(r + 3) + \rho(C_2 \to \pp^r)}{r} \quad \Rightarrow \quad \rho(C_2 \to \pp^r) \geq -r^2 + 8r - 3.\]
As before this implies
\[(r + 1) d_2 - (r - 3)(g_2 - 1) - (r - 1)n \geq 4.\]
So exchanging indices, this falls into one of the previous two cases.
(\hyperlink{one-linnorm}{``Proof when $d_1 = g_1 + r$''} or
\hyperlink{not-linnorm}{``Proof when $d_1 > g_1 + r$, assuming Equation~\eqref{assume} is strict if $(d_1, g_1, r) \in \{(6, 2, 3), (8, 2, 5)\}$''}).
\end{proof}

\section{Proof of Theorem~\ref{small-mid} \label{sec:mid}}

In this section, we prove Theorem~\ref{small-mid}. Our argument will be
by induction on $d$. The basic idea is to inductively degenerate $f$
to a map from a reducible curve,
until we reduce to a case where $\hat{f}$ can be constructed by hand
as the complete linear series attached to the dualizing sheaf of an
appropriate reducible curve.

\begin{proof}[Proof of Theorem~\ref{small-mid} when $a = r$:]
By Theorem~\ref{cor14},
$f(C)$ passes through a set $\Gamma \subset \pp^r$ of
$a + 2 = r + 2$ general points.
Note that $f$ satisfies $\mathbb{H}^1(\mathcal{N}_f) = 0$
by Lemma~\ref{lm:inter}, and that the image of the subscheme
$2 \Gamma \subset C$ under $f$ spans $\pp^r$ (indeed $\Gamma$ spans $\pp^r$).
Again by Theorem~\ref{cor14}, we may find a map $\pp^1 \to \pp^r$
of degree $r$
passing through $\Gamma$.
Taking $\hat{f}$ to be the induced map
$C \cup_\Gamma \pp^1 \to \pp^r$ completes the proof
(since this is an interior BN-curve by Theorem~\ref{main}).
\end{proof}

\begin{proof}[Proof of Theorem~\ref{small-mid} when
$a < r$ and $\rho(d, g, r) \geq r + 1$:]
Note that $\rho(d - 1, g , r) = \rho(d, g, r) - r  - 1 \geq 0$.
We may therefore
let $f_0 \colon C_0 \to \pp^r$
be a general BN-curve of degree $d - 1$ and genus $g$ to $\pp^r$.
Take $\Gamma_0 \subset C_0$ to be a general set of $a + 1$
points.
Write $\Lambda$ for the linear span of the image of
$\Gamma_0 \subset C_0$ under $f_0$. Note that this is a proper subspace
of $\pp^r$ since $\# \Gamma_0 = a + 1 \leq r$ by assumption,
and note that $f_0(\Gamma_0)$ is a collection of points in linear general
position in $\Lambda$ (using our characteristic zero assumption).
Because
\[\# \Gamma_0 = a + 1 \geq \frac{r - 2}{2} + 1 = \frac{r}{2}\]
(and again using our characteristic zero assumption),
the linear span of the image of $2 \Gamma_0 \subset C_0$ under $f_0$
is either a hyperplane (if equality holds above), or all of $\pp^r$
(otherwise).
In the former case, let $H$ be that hyperplane; in the latter,
pick an arbitrary hyperplane $H$ containing $\Lambda$.

Take $R \subset \Lambda$ to be a rational normal curve (of degree $a$)
through $f_0(\Gamma_0)$. Let $L \subset \pp^r$ be a line through
general points $p \in R$ and $q \in f_0(C_0)$.
Consider the map $(C_0 \cup_{\{q\}} L) \cup_{\Gamma \cup \{p\}} R \to \pp^r$.
Writing this map as $(C_0 \cup_\Gamma R) \cup_{\{p,q\}} L \to \pp^r$,
we see by two applications of Theorem~\ref{main-sp} that it is a BN-curve.
Moreover, since $f_0$ is nondegenerate, $q \notin H$; therefore,
the image of $2(\Gamma \cup \{p\}) \subset C_0 \cup_{\{q\}} L$
spans $\pp^r$.
Finally, writing $f \colon C_0 \cup_{\{q\}} L \to \pp^r$,
we conclude by Lemmas~\ref{lm:inter}, \ref{union-res}, and~\ref{lm:union-smooth}
that $\mathbb{H}^1(\mathcal{N}_f) = 0$, completing the proof.
\end{proof}

\begin{proof}[\hypertarget{glarge} Proof of Theorem~\ref{small-mid} when $g \geq r + 1$:]
Note that $\rho(d - r, g - r - 1, r) = \rho(d, g, r) \geq 0$.
We may therefore
let $f_0 \colon C_0 \to \pp^r$
be a general BN-curve of degree $d - r$ and genus $g - r - 1$ to $\pp^r$.
By our inductive hypothesis,
there exists a BN-curve $\hat{f}_0 \colon C_0 \cup_\Gamma \pp^1 \to \pp^r$
with $\# \Gamma = a + 2$ and $\hat{f}_0|_{\pp^1}$ of degree $a$, such that
$\hat{f}_0|_C = f_0$,
and such that the image of $2 \Gamma \subset C_0$ under $\hat{f}_0$
spans $\pp^r$.

By Theorem~\ref{cor14},
$f_0(C_0)$ passes through a set $\Delta \subset \pp^r$ of
$r + 2$ general points (disjoint from~$\Gamma$).
Again by Theorem~\ref{cor14}, we may find a map $f_1 \colon \pp^1 \to \pp^r$
of degree $r$ passing through $\Delta$.
Write $f \colon \pp^1 \cup_\Delta C_0 \to \pp^r$
for the map induced by gluing $f_0$ to $f_1$ along $\Delta$,
which is of degree $d$ from a curve of genus $g$.
By Lemmas~\ref{lm:inter}, \ref{union-res}, and~\ref{lm:union-smooth},
we have $\mathbb{H}^1(\mathcal{N}_f) = 0$.

Finally, let $\hat{f} \colon (\pp^1 \cup_\Delta C_0) \cup_\Gamma \pp^1 \to \pp^r$
be the map obtained by gluing $f$ to $\hat{f}_0$ along $C_0$.
By induction and then Theorem~\ref{main-sp},
$\hat{f} \colon \pp^1 \cup_\Delta (C_0 \cup_\Gamma \pp^1) \to \pp^r$
is a BN-curve. Since the image of
$2 \Gamma \subset C_0 \subset \pp^1 \cup C_0$ under $\hat{f}$
spans $\pp^r$ by construction, this completes the proof.
\end{proof}

\begin{proof}[Proof of Theorem~\ref{small-mid} when $\rho(d, g, r) \geq 1$ and $g \geq 1$ and $a \geq r + 1 - \rho(d, g, r)$:]
By assumption, $\rho(d - 1, g - 1, r) = \rho(d, g, r) - 1 \geq 0$.
We may therefore
let $f_0 \colon C_0 \to \pp^r$
be a general BN-curve of degree $d - 1$ and genus $g - 1$ to $\pp^r$.
We then proceed as in the previous case (\hyperlink{glarge}{``Proof of Theorem~\ref{small-mid} when $g \geq r + 1$''}), taking $\Delta$ to be of size $2$
and $f_1$ to be of degree~$1$.
\end{proof}

\begin{proof}[Completion of proof of Theorem~\ref{small-mid}:]
By what has been proven above, it remains to prove Theorem~\ref{small-mid}
when all of the following conditions are satisfied:
\begin{enumerate}
\item \label{aa} $\rho(d, g, r) = 0$ or $g = 0$ or $a = r - \rho(d, g, r)$;
\item $g \leq r$;
\item $\rho(d, g, r) \leq r$;
\item and $a < r$.
\end{enumerate}
Note that $\rho(d, g, r) \equiv g$ mod $r + 1$;
since $g \leq r$ and $\rho(d, g, r) \leq r$, this implies $\rho(d, g, r) = g$.
In particular, $\rho(d, g, r) = 0$ if and only if $g = 0$.
If $\rho(d, g, r) = g = 0$, the assumptions of Theorem~\ref{small-mid}
imply $a \geq r$, contradicting our assumption here that $a < r$.
Consequently, we must have $a = r - \rho(d, g, r)$ in Condition~\ref{aa}.
Thus, $g = \rho(d, g, r) = r - a$, or, upon rearrangement, $(d, g) = (2r - a, r - a)$.

Let $C$ be a curve of genus $r - a$ of general moduli, and $\Gamma \subset C$ a set of $a + 2$
general points. Identifying $\Gamma$ with any set of $a + 2$ (distinct) points in $\pp^1$,
we construct the curve $C \cup_\Gamma \pp^1$ of genus $r + 1$.
Write $\hat{f} \colon C \cup_\Gamma \pp^1 \to \pp^r$
for map given by the complete linear series attached to the dualizing sheaf $\omega_{C \cup_\Gamma \pp^1}$,
and let $f = \hat{f}|_C$.
Evidently $\hat{f}$ is a BN-curve, so it remains to show that
$\mathbb{H}^1(\mathcal{N}_f) = 0$, and
that the image of the subscheme $2 \Gamma \subset C$ under $f$ spans $\pp^r$.
For these, we use (the long exact sequence in cohomology attached to) the exact sequence of sheaves
\[0 \to K_{\pp^1} \to \omega_{C \cup_\Gamma \pp^1} \to K_C(\Gamma) \to 0.\]
Since $H^1(K_C(\Gamma)) = H^0(K_{\pp^1}) = 0$ and $\dim H^1(K_{\pp^1}) = 1 = \dim H^1(\omega_{C \cup_\Gamma \pp^1})$,
we conclude $H^0(\omega_{C \cup_\Gamma \pp^1}) \to H^0(K_C(\Gamma))$
is an isomorphism.
In particular, $f$ is the complete linear series for the line bundle $K_C(\Gamma)$.
Since $\Gamma \subset C$ is general,
and $\# \Gamma = a + 2 \geq r - a = \text{genus}(C)$,
the line bundle $K_C(\Gamma)$ is a general line bundle of degree $2r - a$ on $C$.
In particular,
$f$ is general BN-curve, so $\mathbb{H}^1(\mathcal{N}_f) = 0$
by Lemma~\ref{lm:inter}.
It remains to show the image of the subscheme $2 \Gamma \subset C$ under $f$ spans $\pp^r$.
Since $f$ is the complete linear series for $K_C(\Gamma)$, this reduces to showing
$H^0(K_C(\Gamma)(-2\Gamma)) = H^0(K_C(-\Gamma)) = 0$,
which again follows from $\# \Gamma = a + 2 \geq r - a = \text{genus}(C)$.
This completes the proof of Theorem~\ref{small-mid}.
\end{proof}

\section{Proof of Theorem~\ref{small-hyp} \label{sec:hyp}}

In the proof of Theorem~\ref{small-hyp}, we write $g'$ and $g''$ for the genera of $C$ and $D$
respectively, and $d'$ and $d''$ for the degrees of $f$ and $g$ respectively.
Our argument will be by induction on $d''$, and for fixed $d''$ by induction on $n$,
via degeneration.
The following lemmas will be useful for degenerating $f$:

\begin{lm} \label{lm:r1}
Let $f \colon C \to \pp^r$ be a general BN-curve of degree $d$ and genus $g$, and $\Gamma$
be a set of $n \leq \max(d, r + 1)$ points in a general hyperplane section of $C$,
and $\Delta \subset C$ be a divisor of degree $m \leq r + 3 - n$ with general support.
Then $\mathbb{H}^1(\mathcal{N}_f(-\Gamma - \Delta)) = 0$.

In particular, there exists a BN-curve
$f \colon C \to \pp^r$ transverse to $H$, of degree~$d$ and genus~$g$,
through $\Gamma \cup \Delta$,
where $\Gamma \subset H$ is a set of $n \leq r + 1$ general points,
and $\Delta \subset \pp^r$ is a set of $m \leq r + 3 - n$ general points,
if and only if $d \geq n$.
\end{lm}
\begin{proof}
We argue by induction on $g$. By increasing $m$ if necessary,
we note that $\oo_C(\Gamma + \Delta)$ is a general line bundle of degree $n + m$
if $g \in \{0, 1\}$; the result thus follows from Theorem~1.3 of~\cite{aly} in this case.

For the inductive step, we first consider the case where $\rho(d, g, r) \geq 1$
and $g \geq 2$. These imply $\rho(d - 1, g - 1, r) = \rho(d, g, r) - 1 \geq 0$
and $g - 1 \geq 1$.
By Theorem~\ref{main-sp}, we can specialize
$f$ to a map $f^\circ \colon C_0 \cup_{\{p, q\}} \pp^1 \to \pp^r$,
where $f^\circ|_{C_0}$ is of degree $d - 1$ and genus $g - 1$,
and $f^\circ|_{\pp^1}$ is a line.
Since our inequalities imply $d - 1 \geq r + 1$, we can specialize $\Gamma \cup \Delta$
to lie on $C_0$. By induction, we have $\mathbb{H}^1(\mathcal{N}_{f^\circ|_{C_0}}(-\Gamma - \Delta)) = 0$;
by Lemma~\ref{lm:inter}, we have $\mathbb{H}^1(\mathcal{N}_{f^\circ|_{\pp^1}}) = 0$.
Applying Lemmas~\ref{union-res} and~\ref{lm:union-smooth}, we conclude that
$\mathbb{H}^1(\mathcal{N}_{f^\circ}(-\Gamma - \Delta)) = 0$ as desired.

Next we consider the case where
$\rho(d, g, r) = 0$ and $g \geq 2$, which implies $g \geq r + 1$.
Write $a = \lceil (r - 2) / 2 \rceil$.
Note that $\rho(d - a, g - a - 1, r) = \rho(d, g, r) + r - a \geq 0$,
and that $g - a - 1 \geq 1$.
By Theorem~\ref{small-mid} and Lemma~\ref{cohom-small-mid},
we can specialize
$f$ to $f^\circ \colon C_0 \cup_A \pp^1 \to \pp^r$, with $\# A = a + 2$,
where $f^\circ|_{C_0}$ is of degree $d - a$ and genus $g - a - 1$,
and $f^\circ|_{\pp^1}$ is of degree $a$,
satisfying
$\mathbb{H}^1(\mathcal{N}_{f^\circ}|_{\pp^1}(-A)) = 0$.
Since our inequalities imply $d - a \geq r + 1$,
we can specialize $\Gamma \cup \Delta$
to lie on $C_0$.
By induction, we have $\mathbb{H}^1(\mathcal{N}_{f_0|_{C_0}}(-\Gamma - \Delta)) = 0$.
Applying Lemmas~\ref{union-res} and~\ref{lm:union-smooth}, we conclude that
$\mathbb{H}^1(\mathcal{N}_{f_0}(-\Gamma - \Delta)) = 0$ as desired.
\end{proof}

\begin{lm} \label{admits-deformation}
Let $f \colon C \cup_\Gamma D \to \pp^r$ be an unramified map from a reducible curve,
such that $f|_D$ factors as a composition of a general BN-curve
$f_D \colon D \to H$ of degree $d$ and genus $g$
with the inclusion of a hyperplane $H \subset \pp^r$,
while $f|_C$ is general in some component of
the space of WBN-curves transverse to $H$ along $\Gamma$.
Let $\Delta$ be a set of general points on $D$,
and $\Delta' \subset f(C) \cap H \smallsetminus \Gamma$,
such that $\Gamma \cup \Delta'$ and $\Delta$
are general sets of points in $H$.
Write $n = \# \Gamma$ and $m = \# \Delta$.

If $d - g + n \geq \max(m, r - 1)$,
then $H^1(f|_D^* \oo_{\pp^r}(1)(\Gamma)) = 0$, and
there exists a deformation of $f$
still passing through $\Delta \cup \Delta'$, and transverse
to $H$ along $\Delta \cup \Delta'$.
\end{lm}
\begin{proof}
First we show $H^1(f|_D^* \oo_{\pp^r}(1)(\Gamma)) = H^1(f|_D^* \oo_{\pp^r}(1)(\Gamma - \Delta - x)) = 0$ for any $x \in \Delta$.
Note that $\Gamma$, $\Delta \smallsetminus \{x\}$, and $\{x\}$,
are general subsets of $n$, $m - 1$, and $1$ points on $D$ respectively.
If $\mathcal{L}$ is a line bundle on a curve $X$, and $p \in X$ a general point,
and $k$ a positive integer, then by \cite{wronskian}:
\[\dim H^0(\mathcal{L}(-kp)) = \max(0, \dim H^0(\mathcal{L}) - k),\]
and thus by Serre duality:
\[\dim H^1(\mathcal{L}(kp)) = \max(0, \dim H^1(\mathcal{L}) - k).\]
Applying this for every point in $\Gamma$, then
every point in $\Delta \smallsetminus \{x\}$,
and then for $x$, it suffices to show
\[\dim H^1(f|_D^* \oo_{\pp^r}(1)) \leq \# \Gamma = n \tand \chi(f|_D^* \oo_{\pp^r}(1)(\Gamma - \Delta - x)) \geq 0.\]
But we have, by assumption,
\[\chi(f|_D^* \oo_{\pp^r}(1)(\Gamma - \Delta - x)) = (d + 1 - g) + n - m - 1 = d - g + n - m \geq 0.\]
Moreover, since $f_D$ is a general BN-curve,
we know by counting dimensions (using the Brill--Noether theorem \cite{bn})
that it is either nonspecial or linearly normal;
i.e.\ that $H^1(f|_D^* \oo_{\pp^r}(1)) = 0$ or
$H^0(f|_D^* \oo_{\pp^r}(1)) = r$.
In the first case, $H^1(f|_D^* \oo_{\pp^r}(1)) = 0 \leq \# \Gamma$, while
in the second case,
\[\dim H^1(f|_D^* \oo_{\pp^r}(1)) = r - \chi(f|_D^* \oo_{\pp^r}(1)) = r - (d + 1 - g) \leq n.\]
This shows $H^1(f|_D^* \oo_{\pp^r}(1)(\Gamma)) = H^1(f|_D^* \oo_{\pp^r}(1)(\Gamma - \Delta - x)) = 0$ as desired.

Next we show 
$f$ admits a deformation
still passing through $\Delta \cup \Delta'$, and transverse
to $H$ along $\Delta \cup \Delta'$.
For this it suffices
by deformation theory to check
$H^1(N^x) = 0$ for all $x \in \Delta$
(since $f$ is already transverse to $H$ along $\Delta'$)
where $N^x$ is defined by
\[N^x = \ker \left(N_f(-\Delta - \Delta') \to N_{H/\pp^r}|_{2x}\right).\]
To show this, we use the exact sequences
\begin{gather*}
0 \to N_f|_C (-\Gamma - \Delta') \to N^x \to \ker \left(N_f|_D (-\Delta) \to N_{H/\pp^r}|_{2x}\right) \to 0 \\
0 \to N_{f_D} (-\Delta) \to \ker \left(N_f|_D (-\Delta) \to N_{H/\pp^r}|_{2x}\right) \to f_D^* N_{H/\pp^r} (\Gamma - \Delta - x) \simeq f|_D^* \oo_{\pp^r}(1)(\Gamma - \Delta - x) \to 0.
\end{gather*}
Above we showed $H^1(f|_D^* \oo_{\pp^r}(1)(\Gamma - \Delta - x)) = 0$.
Applying Lemmas~\ref{lm:inter} and~\ref{union-res},
we see
$H^1(N_f|_C (-\Gamma - \Delta')) = 0$.
And by Lemma~\ref{lm:inter}, we have $H^1(N_{f_D} (-\Delta)) = 0$.
We conclude $H^1(N^x) = 0$ as desired.
\end{proof}

\begin{proof}[Proof of Theorem~\ref{small-hyp} when $f$ is degenerate and $d'' \geq g'' + r - 1$:]
As $f$ is degenerate, $n \leq r - 1$. Thus,
\[(r + 1) d'' - rg'' + r \geq (r + 1) (g'' + r - 1) - rg'' + r = g'' + r^2 + r - 1 \geq r(r - 1) \geq rn.\]
The result thus follows from Theorem~\ref{main-sp}.
\end{proof}

\begin{proof}[Proof of Theorem~\ref{small-hyp} when $f$ is nondegenerate, $d' \geq g' + r$, and $d'' \geq g'' + r - 1$:]
If $g' = 0$, the result follows from Theorem~\ref{main-hyp}.
Otherwise, as $d' \geq g' + r$ and $g' \geq 1$, we have
\[\rho(d' - 1, g' - 1, r) = \rho(d, g, r) - 1 \geq 0.\]

By Theorem~\ref{main-sp} and Lemma~\ref{lm:r1}, we can specialize
$f$ to $f^\circ \colon C_0 \cup_{\{p, q\}} \pp^1 \to \pp^r$,
where $f^\circ|_{C_0}$ is of degree $d - 1$ and genus $g - 1$,
and $f^\circ|_{\pp^1}$ is a line,
while still passing through a set $\Gamma = \Gamma_0 \cup \{x\}$
of $n$ general points;
more precisely, where
$f^\circ|_{C_0}$ passes through the set $\Gamma_0$ of $n - 1$ general points,
and $f^\circ|_{\pp^1}$ passes through $x$.
By Lemma~\ref{can-degenerate-small-mid}, it suffices to show
$C_0 \cup_{\Gamma_0 \cup \{p, q\}} (\pp^1 \cup_x D) \to \pp^r$
is a BN-curve.

Moreover, by Lemma~\ref{lm:r1}, any deformation of $\Gamma_0 \cup \{p, q\}$
lifts to a deformation of $f^\circ|_{C_0}$.
Since $\pp^1 \cup_x D \to \pp^r$
is an BN-curve by Theorem~\ref{main-sp},
whose degree $d'' + 1$ and genus $g''$ satisfy $d'' + 1 \geq g'' + r$
by assumption, $\pp^1 \cup_x D \to \pp^r$ is an NNS-curve.
By Theorem~\ref{cor14}, we can thus deform $\pp^1 \cup_x D \to \pp^r$
to pass through $n + 1$ general points.
Theorem~\ref{main} then implies $C_0 \cup_{\Gamma_0 \cup \{p, q\}} (\pp^1 \cup_x D) \to \pp^r$
is a BN-curve, as desired.
\end{proof}

\begin{proof}[Proof of Theorem~\ref{small-hyp} when $f$ is nondegenerate, $d' \leq g' + r - 1$, and $d'' \geq g'' + r - 1$:]
If $n \leq r$, the result follows from Theorem~\ref{main-sp}. If $n = r + 1$
and $d'' = r - 1$, the result follows from Theorem~\ref{small-mid}.
We may thus suppose $n \geq r + 1$ with strict inequality if $d'' = r - 1$.

Since $d' \leq g' + r - 1$ and $\rho(d', g', r) \geq 0$,
we have $g' \geq r + 1$.
By our inductive hypothesis,
we can specialize $f$ to $f^\circ \colon C_0 \cup_A \pp^1 \to \pp^r$,
with $\# A = r + 1$,
where $f^\circ|_{C_0}$ is of degree $d - r + 1$ and genus $g - r$,
and $f^\circ|_{\pp^1}$ is of degree $r - 1$,
while still passing through a set $\Gamma = \Gamma_0 \cup \{x\}$
of $n$ general points;
more precisely, where
$f^\circ|_{C_0}$ passes through the set $\Gamma_0$ of $n - 1$ general points
(by Lemma~\ref{lm:r1}),
and $f^\circ|_{\pp^1}$ passes through $x$.

By Lemma~\ref{can-degenerate-small-hyp-f}, it suffices to show
$(C_0 \cup_{\Gamma_0} D) \cup_{A \cup \{x\}} \pp^1 \to \pp^r$ is a BN-curve.
But this holds by induction, unless $n = r + 2$ and $d'' = r - 1$.

If $n = r + 2$ and $d'' = r - 1$, our assumption that $C \cup_\Gamma D \to \pp^r$
has nonnegative Brill--Noether number rearranges to $\rho(d' - r, g' - r, r) \geq 1$.
We can therefore repeat the same construction as above with 
$f^\circ|_{C_0}$ of degree $d - r$ and genus $g - r$,
and $f^\circ|_{\pp^1}$ of degree $r$,
using Theorem~\ref{main} in place of our inductive hypothesis.
By Lemma~\ref{can-degenerate-small-hyp-f}, it suffices to show
$(C_0 \cup_{\Gamma_0} D) \cup_{A \cup \{x\}} \pp^1 \to \pp^r$ is a BN-curve.
But this holds by induction and Theorem~\ref{main}.
\end{proof}

\begin{proof}[Proof of Theorem~\ref{small-hyp} when $d'' \leq g'' + r - 2$:]
Since $\rho(d'', g'', r - 1) \geq 0$ and $d'' \leq g'' + r - 2$, we have $g'' \geq r$.
By Theorem~\ref{main}, we may specialize $g$ to $g^\circ \colon D_0 \cup_\Delta \pp^1 \to H$,
where $\Delta$ is a general set of $r + 1$ points,
$g^\circ|_{D_0}$ is of degree $d'' - r + 1$ and genus $g'' - r$,
and $g^\circ|_{\pp^1}$ is of degree $r - 1$. We specialize $\Gamma$ onto $D_0$;
this can be done so $\Gamma$ remains general by Lemma~\ref{lm:r1}.
By Lemma~\ref{can-degenerate-small-hyp-g}, it suffices to show
$(C \cup_\Gamma D_0) \cup_\Delta \pp^1 \to \pp^r$ is a BN-curve
and 
$\mathbb{H}^1(\mathcal{N}_{g^\circ}) = H^1((g^\circ)^* \oo_{\pp^r}(1)(\Gamma)) = 0$.
To see $\mathbb{H}^1(\mathcal{N}_{g^\circ}) = 0$, we apply Lemmas~\ref{lm:inter}, \ref{union-res}, and~\ref{lm:union-smooth};
to see $H^1((g^\circ)^* \oo_{\pp^r}(1)(\Gamma)) = 0$,
we apply Lemma~\ref{admits-deformation}.

Finally, to show $(C \cup_\Gamma D_0) \cup_\Delta \pp^1 \to \pp^r$ is a BN-curve,
we can apply our inductive hypothesis twice, since
$h := C \cup_\Gamma D_0 \to \pp^r$
admits a deformation passing through $\Delta$
which is transverse to $H$ along $\Delta$
by Lemma~\ref{admits-deformation}.
\end{proof}

\begin{proof}[Proof of Corollary~\ref{cor:small-hyp}.] We argue by induction on $m$;
the base case of $m = 0$ is just Theorem~\ref{small-hyp}.
For the inductive step, the same argument as in Lemma~\ref{can-degenerate-small-hyp-f-var}
implies that it suffices to show the resulting curve is a BN-curve
after we specialize $h_1$ to factor through $H$.
Applying our inductive hypothesis, we conclude that
$C \cup_\Gamma (D \cup_{p_1} \pp^1) \cup_{p_2} \pp^1 \cdots \cup_{p_m} \pp^1 \to \pp^r$
is a BN-curve as desired.
\end{proof}

\bibliographystyle{amsplain.bst}
\bibliography{mrcbib}

\end{document}